\documentclass{siamltex}
\usepackage{amsmath,amsfonts,amssymb}
\usepackage{latexsym}
\usepackage{graphicx,overpic}
\usepackage{subfigure,caption}
\usepackage{pgfplots}
\pgfplotsset{compat=newest}
\usepackage{color}
\usepackage{booktabs}
\usepackage{caption}
\usepackage{lscape}

\usepackage{tikz}

\newcommand{\R}{\mathbb R}
\newcommand{\bA}{\mathbf A}
\newcommand{\bB}{\mathbf B}

\newcommand{\bH}{\mathbf H}
\newcommand{\bI}{\mathbf I}

\newcommand{\bP}{\mathbf P}

\newcommand{\bV}{\mathbf V}

\newcommand{\bg}{\mathbf g}
\newcommand{\bn}{\mathbf n}
\newcommand{\be}{\mathbf e}

\newcommand{\bu}{\mathbf u}

\newcommand{\bv}{\mathbf v}
\newcommand{\bw}{\mathbf w}

\newcommand{\bbf}{\mathbf f}

\newcommand{\T}{\mathcal T}

\newcommand{\tr}{{\rm tr}}

\newcommand{\divG}{{\mathop{\,\rm div}}_{\Gamma}}
\newcommand{\gradG}{\nabla_{\Gamma}}

\newcommand{\gradGh}{\nabla_{\Gamma_h}}

\newcommand{\OGamma}{\Omega^\Gamma_h}

\newcommand{\cT}{\mathcal T}

\newtheorem{assumption}{Assumption}[section]
\newtheorem{remark}{Remark}[section]

\numberwithin{equation}{section}
\begin{document}
\title{Trace Finite Element Methods for Surface Vector-Laplace Equations}
\author{Thomas Jankuhn\thanks{Institut f\"ur Geometrie und Praktische  Mathematik, RWTH-Aachen
University, D-52056 Aachen, Germany (jankuhn@igpm.rwth-aachen.de)} \and
Arnold Reusken\thanks{Institut f\"ur Geometrie und Praktische  Mathematik, RWTH-Aachen
University, D-52056 Aachen, Germany (reusken@igpm.rwth-aachen.de).}
}
\maketitle

\begin{abstract} In this paper  we analyze a class of trace finite element methods (TraceFEM) for the discretization of vector-Laplace equations. A key issue in the finite element discretization of such problems is the treatment of the  constraint that the unknown vector field must be tangential to the surface (``tangent condition''). We study three  different natural techniques for treating the tangent condition, namely a consistent penalty method, a simpler inconsistent penalty method and a Lagrange multiplier method. A main goal of the paper is to present an analysis that reveals important properties of these three different techniques for treating the tangent constraint. A detailed error analysis is presented that takes the approximation of both the geometry of the surface and the solution of the partial differential equation into account. Error bounds in the energy norm are derived that show how the discretization error depends on relevant parameters such as the degree of the polynomials used for the approximation of the solution, the degree of the polynomials used for the approximation of the level set function that characterizes the surface, the penalty parameter and  the degree of the polynomials used for the approximation of Lagrange multiplier.
\end{abstract}
\section{Introduction} 
In recent years there has been a strongly growing interest in the field of modeling and numerical simulation of surface fluids, cf. the papers \cite{arroyo2009,Yavarietal_JNS_2016,Kobaetal_QAM_2017,miura2017singular,Nitschkeetal_arXiv_2018,reuther2015interplay,Jankuhn1}, in which Navier-Stokes type PDEs for (evolving) surfaces with fluidic properties are proposed. Concerning error analysis of numerical methods for surface (Navier-)Stokes equations there  are only very few results available. In \cite{hansbo2016stabilized} an error analysis for a finite element discretization method for surface Darcy equations is presented. First  error analysis results for a finite element discretization method of surface Stokes equations are given in \cite{Olshanskiietal_arXiv_2018}. In that paper a $P1$-$P1$ stabilized trace finite element method is studied and optimal error bounds are derived. 
As far as we know, no other papers on rigorous error analyses of finite element methods for surface (Navier-)Stokes are available in the literature. One crucial point in the development and analysis of finite element methods for surface Stokes equations is the numerical treatment of the  constraint that the flow must be tangential to the surface (``tangent condition''). This constraint also occurs in the class of surface vector-Laplace problems. These problems are easier to handle than the surface Stokes equations because they only contain a velocity unknown and not a presssure variable. Such vector-Laplace problems are a meaningful simplification of the surface Stokes equation for the development and analysis of finite element methods. In two very recent papers \cite{hansbo2016analysis,grossvectorlaplace} finite element methods for the discretization of vector-Laplace problems are analyzed. The topic of the present paper is closely related to the topics treated in \cite{hansbo2016analysis,grossvectorlaplace}. We briefly 
discuss the main results of these two papers. In \cite{hansbo2016analysis} a finite element method based on standard continuous parametric Lagrange elements, in the spirit of the surface finite element method  (SFEM) for scalar surface PDEs, introduced by Dziuk and Elliott \cite{Dziuketal_AN_2013}, is studied. The tangent condition is weakly enforced by a penalization term. Optimal discretization error estimates are derived that take the approximation of both the geometry of the surface and the solution of the partial differential equation into account.  
In \cite{grossvectorlaplace} a different finite element technique, namely the  trace finite element method (TraceFEM) is studied. This TraceFEM has been thoroughly analyzed for \emph{scalar} surface PDEs, cf.  the overview paper \cite{olshanskii2016trace}. In order to satisfy the tangent constraint for vector-Laplace problems,  a Lagrange multiplier approach is proposed and analyzed in \cite{grossvectorlaplace}. Optimal error estimates are derived, which, however, do not take the errors due to the approximation of the geometry of the surface into account.

In this paper we consider the same vector-Laplace problem as in \cite{grossvectorlaplace}, which is similar to the one in \cite{hansbo2016analysis}. We study the TraceFEM and \emph{three} different natural techniques for treating the tangent condition:
\begin{itemize}
 \item A consistent penalty method, which is the same as the one analyzed (for the SFEM) in \cite{hansbo2016analysis}.
 \item An \emph{in}consistent penalty method as introduced in \cite{Jankuhn1}. This method is simpler as the above-mentioned consistent one, because an approximation of the Weingarten map is avoided.
 \item A Lagrange multiplier method as  in \cite{grossvectorlaplace}. 
\end{itemize}
For higher order approximation we use the parametric version of TraceFEM, which,  for scalar surface PDEs, is analyzed in \cite{grande2017higher}. \emph{The main goal of the paper is to present an analysis that reveals important properties of these three different techniques for treating the tangent constraint}. The topics studied in this paper relate to the ones treated in \cite{hansbo2016analysis,grossvectorlaplace} as follows. Different from \cite{hansbo2016analysis}, we study the TraceFEM (instead of SFEM) and we analyze and compare three different techniques for handling the tangent condition. In   \cite{grossvectorlaplace} only the Lagrange multiplier method is treated and errors due to geometry approximation are neglected; in this paper we take geometry errors into account and besides the Lagrange multiplier method we also analyze two penalty methods.

 Since we use TraceFEM, it is necessary to include some stabilization to damp instabilities caused by ``small cuts''. For this we  use the normal derivative volume stabilization, known from the literature \cite{grande2017higher}.  We derive error estimates that  
take the approximation of both the geometry of the surface and the solution of the partial differential equation into account.   The main results of this paper are the discretization error bounds, in the energy norm, given in section~\ref{sectDiscrerror}. These results reveal how the errors depend on relevant parameters $k$, $k_g$, $k_p$, $\eta$, $k_l$. Here  $k$ denotes the degree of the polynomials used for the approximation of the solution, $k_
g$ the degree of the polynomials used for the approximation of the level set function that characterizes the surface, $k_p \geq k_g$ the order of accuracy of the normal vector approximation used in the penalization term (in both penalty methods), $\eta$ the penalty parameter and $k_l$ the degree of the polynomials used for the approximation of the Lagrange parameter (in the third method).
These error bounds lead to several interesting conclusions. For example, for both penalty methods it is necessary to take $k_p \geq k+1$ in order to obtain optimal error bounds. For the SFEM and the consistent penalty method such a result is also derived in \cite{hansbo2016analysis}. For the consistent penalty method one obtains an optimal order error bound of order $\sim  h^k$ if one takes $k_p=k+1$, $k_g=k$ (i.e. isoparametric spaces), $\eta \sim  h^{-2}$. Such an optimal result does not hold for the (simpler) \emph{in}consistent penalty method. Optimal balancing of terms leads to $k_p=k+1$, $k_g=k$, $\eta \sim h^{-(k+1)}$ and an error bound  of order $\sim h^{\frac12(k+1)}$ for the inconsistent penalty method. This bound is optimal (only) for the important case of linear finite elements, i.e., $k=1$. Hence, in that case this simpler method (which avoids approximation of the Weingarten map) may be more attractive than the consistent penalty method. For the Lagrange multiplier method we do not obtain optimal error bounds for the isoparametric 
case $k=k_g$. For $k_g=k+1$ we obtain optimal bounds both for $k_l=k$ and $k_l=k-1$ (if $k\geq 2$). In the latter case one has to take an appropriate scaling for the  parameter used in the normal derivative volume stabilization term (cf. section~\ref{sectDiscrerror} for more explanation).

Results of numerical experiments are presented that illustrate the error behavior predicted by our error analysis. 

In this paper the analysis of the three methods for treating the constraint is done for the TraceFEM. We expect that similar error bounds and conclusions hold if instead of the TraceFEM the SFEM is used. For the consistent penalty method, these results are derived in \cite{hansbo2016analysis}.

As in the paper \cite{hansbo2016analysis} the analysis is rather technical. This is mainly due to the fact that for the vector case we have to derive bounds for the consistency error caused by the geometry approximation. For the TraceFEM such estimates are not available in the literature (in  \cite{hansbo2016analysis} such estimates are derived for the vector analogon of the Dziuk-Elliot SFEM). 

Since the error analysis in the energy norm as presented in this paper is rather long and technical we do not include an $L^2$-error analysis. This will be addressed in future work.
Based on  the results obtained for the vector-Laplace problem we plan to analyze these methods applied to surface Stokes equations. This a topic of current research. 

The remainder of the paper is organized as follows. In section~\ref{Sectcontinuous} we introduce the variational formulation of the surface vector-Laplace problem that we consider, and give three  related formulations (two of penalty type and one based on a Lagrange multiplier) in which the tangent constraint is treated differently. In section~\ref{sectParametric}  we collect properties of a parametric finite element space known from the literature. Based on this space and the three different variational formulations we define corresponding TraceFEM discrete problems in section~\ref{sectFEmethods}. An error analysis of these methods is presented in section~\ref{sectAnalysis}. The structure of this analysis is along the usual lines. We first derive discrete stability results and based on these formulate Strang Lemmas, in which the energy norm of the discretization error is bounded by a sum of an approximation error and a consistency error. Bounds of the approximation error are easy to derive, based on results 
known from the literature. For proving satisfactory bounds for the consistency term we need a longer and tedious analysis. Finally, in section~\ref{sectExperiments} we present results of numerical experiments.

\section{Continuous problem} \label{Sectcontinuous}
We assume that $\Omega \subset \mathbb{R}^3$ is a polygonal domain which contains a connected compact smooth hypersurface $\Gamma$ without boundary. A tubular neighborhood of $\Gamma$ is denoted by
\begin{equation*}
U_\delta := \left\lbrace x \in \mathbb{R}^3 \mid \vert d(x) \vert < \delta \right\rbrace,
\end{equation*}
with $\delta > 0$ and $d$ the signed distance function to $\Gamma$, which we take negative in the interior of $\Gamma$. The surface $\Gamma$ is the zero level of a smooth level set function $\phi \colon U_\delta \to \mathbb{R}$, i.e. $\Gamma = \{ x \in \Omega \mid \phi(x) =0 \}$. This level set function is not necessarily close to a distance function but has the usual properties of a level set function:
\begin{equation*}
\Vert \nabla \phi(x) \Vert \sim 1, \quad \Vert \nabla^2 \phi(x) \Vert \leq c \quad \text{for all } x \in U_\delta.
\end{equation*}
We assume that the level set function   $\phi$  is sufficiently smooth. On $U_\delta$ we define $\bn(x) = \nabla d(x)$, the outward pointing unit normal on $\Gamma$, $\bH(x) = \nabla^2d(x)$, the Weingarten map,  $\bP = \bP(x):= \bI - \bn(x)\bn(x)^T$, the orthogonal projection onto the tangential plane, $p(x) = x - d(x)\bn(x)$, the closest point projection. We assume $\delta>0$ to be sufficiently small such that the decomposition 
\begin{equation*}
x = p(x) + d(x) \bn(x)
\end{equation*}
is unique for all $x \in U_{\delta}$. The constant normal extension for vector functions $\bv \colon \Gamma \to \mathbb{R}^3$ is defined as $\bv^e(x) := \bv(p(x))$, $x \in U_{\delta}$. The extension for scalar functions is defined similarly. Note that on $\Gamma$ we have $\nabla \bw^e = \nabla(\bw \circ p) = \nabla \bw^e \bP$, with $\nabla \bw := (\nabla w_1, \nabla w_2, \nabla w_3)^T \in \mathbb{R}^{3 \times 3}$ for smooth vector functions $\bw \colon U_\delta \to \mathbb{R}^3$. For a scalar function $g \colon U_\delta \to \mathbb{R}$ and a vector function $\bv \colon U_\delta \to \mathbb{R}^3$ the covariant derivatives are defined by
\begin{equation*} \begin{split}
\nabla_{\Gamma} g(x) = \bP(x)\nabla g(x), \quad x \in \Gamma, \qquad
\nabla_{\Gamma} \bv(x) = \bP(x)\nabla \bv(x) \bP(x), \quad x \in \Gamma. 
\end{split}
\end{equation*} 
On $\Gamma$ we consider the surface stress tensor (see \cite{GurtinMurdoch75}) given by
\begin{equation*}
E_s(\bu):= \frac{1}{2} \left( \gradG \bu + \gradG^T \bu \right),
\end{equation*}
with $\gradG^T \bu := (\gradG \bu)^T$. To simplify the notation we write $E=E_s$.  The surface divergence operator for vector-valued functions $\bu \colon \Gamma \to \mathbb{R}^3$ and tensor-valued functions $\bA \colon \Gamma \to \mathbb{R}^{3\times3}$ are defined as
\begin{equation*} \begin{split}
\divG \bu := \textrm{tr} (\gradG \bu),  \qquad
\divG \bA := \left( \divG (\be_1^T\bA), \divG (\be_2^T\bA),\divG (\be_3^T\bA)  \right)^T,
\end{split}
\end{equation*}
with $\be_i$ the $i$th basis vector in $\mathbb{R}^3$. For a given force vector $\bbf \in L^2(\Gamma)^3$, with $\bbf \cdot \bn =0$, we consider the following elliptic \emph{vector Laplace problem}: determine $\bu \colon \Gamma \to \R^3$ with $\bu \cdot \bn = 0$ and 
\begin{equation} \label{eqstrong}
- \bP \divG (E(\bu)) + \bu = \bbf \qquad \text{on } \Gamma.
\end{equation}
We added the zero order term on the left-hand side to avoid technical details related to the kernel of the strain tensor $E$ (the so-called Killing vector fields).
The surface Sobolev space of weakly differentiable vector valued functions is denoted by
\begin{equation} \label{eqdefH1}
\begin{gathered}
\bV:= H^1(\Gamma)^3, \quad \text{with} ~ \Vert \bu \Vert_{H^1(\Gamma)}^2 := \int_{\Gamma} \Vert \bu(s) \Vert_2^2 + \Vert \nabla \bu^e(s) \Vert_2^2 \, ds.
\end{gathered}
\end{equation}
Note that $\|\nabla \bu^e\|_2=\|(\nabla \bu^e)^T\|_2$ and  on $\Gamma$ we have 
\begin{equation} \label{defH1} (\nabla \bu^e)^T=\bP\big(\nabla u_1^e, \nabla u_2^e, \nabla u_3^e \big) = \sum_{i=1}^3 \gradG u_i \be_i^T .
\end{equation}
 Hence, the norm in \eqref{eqdefH1} is a natural extension to vector valued functions of the usual scalar $H^1(\Gamma)$-norm. The corresponding space of tangential vector field is denoted by
\begin{equation*}
\bV_T := \left\lbrace \bu \in \bV \mid \bu \cdot \bn =0 \right\rbrace.
\end{equation*}
A vector $\bu \in \bV$ can be  decomposed into a tangential and a normal part. We use the notation:
\begin{equation*}
\bu = \bP \bu + (\bu \cdot \bn)\bn = \bu_T + u_N\bn.
\end{equation*}
For $\textbf{u}, \textbf{v} \in \bV$ we introduce the bilinear form
\begin{equation*} \begin{split}
a(\textbf{u}, \textbf{v}) := \int_\Gamma E(\bu) : E(\bv) \, ds + \int_\Gamma \textbf{u} \cdot \textbf{v} \, ds.
\end{split}
\end{equation*}
For a given $\bbf \in L^2(\Gamma)^3$ with $\bbf \cdot \bn = 0$ we consider the following weak formulation of \eqref{eqstrong}: determine $\bu = \bu_T \in \bV_T$ such that 
\begin{equation} \tag{C}\label{contform}
a(\bu_T, \bv_T) = (\bbf, \bv_T)_{L^2(\Gamma)} ~~~ \text{for all}~ \bv_T \in \bV_T.
\end{equation} 
The bilinear form $a(\cdot,\cdot)$ is continuous on $\bV_T$. The ellipticity of $a(\cdot,\cdot)$ on $\bV_T$ follows from the following surface Korn inequality, which is derived in \cite{Jankuhn1}. 
\begin{lemma} \label{lemmakorn}
Assume $\Gamma$ is $C^2$ smooth. There exists a constant $c_K > 0$ such that 
\begin{equation*}
\Vert \bu \Vert_{L^2(\Gamma)} + \Vert E(\bu) \Vert_{L^2(\Gamma)} \geq c_K \Vert \bu \Vert_{H^1(\Gamma)} \qquad \text{for all } \bu \in \bV_T.
\end{equation*}
\end{lemma}
Hence, the weak formulation \eqref{contform} is a well-posed problem. The unique solution is denoted by $\bu^* = \bu_T^*$.

    The weak formulation \eqref{contform} is not very suitable for a finite element discretization, because we would need vector finite element functions that are tangential to $\Gamma$. Obvious alternatives are obtained by allowing general (not necessarily tangential) vector functions $\bu$ and to treat the constraint $\bu \cdot \bn=0$ by either a penalty approach or a Lagrange multiplier. For vector Laplace problems these were considered in the recent papers \cite{hansbo2016analysis,grossvectorlaplace}. Below we introduce two different penalty formulations and a Lagrange multiplier formulation. These formulations are the basis for (higher order) Galerkin finite element methods that are defined in section~\ref{sectFEmethods}. The remainder of the paper is then devoted to an error analysis of these methods.

Define
\begin{gather*}
\bV_* := \left\lbrace \bu \in L^2(\Gamma)^3 \mid \bu_T \in \bV_T, u_N \in L^2(\Gamma) \right\rbrace, 
\quad \Vert  \bu \Vert_{V_*}^2 := \Vert \bu_T \Vert_{H^1(\Gamma)}^2 + \Vert u_N \Vert_{L^2(\Gamma)}^2.
\end{gather*} 
Using the  identity (for $\bu \in \bV$)
\begin{equation} \label{identi}
E(\bu) = E(\bu_T) + u_N \bH
\end{equation}
we introduce, with some abuse of notation, the bilinear form 
\begin{equation} \label{defbla}
a(\textbf{u}, \textbf{v}) := \int_\Gamma (E(\bu_T) + u_N \bH) : (E(\bv_T) + v_N \bH) \, ds + \int_\Gamma \textbf{u} \cdot \textbf{v} \, ds, \quad \bu,\bv \in \bV_*.
\end{equation}
This bilinear form is well-defined and continuous on $\bV_*$. We also define the penalty bilinear form
\[
 k(\bu,\bv) := \eta \int_\Gamma  (\textbf{u} \cdot \textbf{n}) ~ (\textbf{v} \cdot \textbf{n})  \, ds,
\]
with $\eta >0$ a penalty parameter. 
The first  penalty approach that we introduce (also considered in \cite{grossvectorlaplace}) is as follows: for a given $\bbf \in L^2(\Gamma)^3$ with $\bbf \cdot \bn = 0$ determine $\textbf{u} \in \bV_*$, such that
\begin{equation} \tag{P1}\label{projectedcontform0}
 a(\bu, \bv) + k(\bu, \bv) = ( \textbf{f}, \textbf{v} )_{L^2(\Gamma)} \quad \text{for all}~~\textbf{v} \in \bV_*.
\end{equation}
One can easily check that for $\eta$ sufficiently large we have an ellipticity property: there exists a constant $c>0$ such that
\begin{equation} \label{ell1}
 a(\bu, \bu) + k(\bu, \bu) \geq c \|\bu\|_{V_\ast}^2 \quad \text{for all}~~\bu \in V_\ast.
\end{equation}
Furthermore, $a(\cdot,\cdot) +k(\cdot,\cdot)$ is continuous on $V_\ast$. Hence, for $\eta$ sufficiently large the problem \eqref{projectedcontform0} is well-posed. 
 The formulation, however, is \emph{inconsistent}. 
\begin{lemma} \label{lemconsist} Take $\eta$ sufficiently large such that \eqref{ell1} holds. For the unique solution $\bu$ of \eqref{projectedcontform0}  the following holds:
\begin{equation} \label{penerror}
 \|\bu_T - \bu_T^\ast\|_{H^1(\Gamma)} + \Vert u_N \Vert_{L^2(\Gamma)}  \leq c \eta^{-1} \Vert \bbf \Vert_{L^2(\Gamma)}.
\end{equation}
\end{lemma}
\begin{proof} There exists a constant $\tilde{c}>0$ such that for the solution of problem \eqref{projectedcontform0} we have $\Vert \bu \Vert_{V_*} \leq \tilde{c} \Vert \bbf \Vert_{L^2(\Gamma)}$. Testing problem \eqref{projectedcontform0} with $\bv = u_N \bn$ we obtain $
a(\bu, u_N \bn) + \eta \Vert u_N \Vert_{L^2(\Gamma)}^2 = 0$. 
Using the Cauchy-Schwarz inequality we get
\begin{equation*}
\eta \Vert u_N \Vert_{L^2(\Gamma)}^2 \leq C \Vert \bu \Vert_{V_*} \Vert u_N \Vert_{L^2(\Gamma)} \leq \tilde{C} \Vert \bbf \Vert_{L^2(\Gamma)} \Vert u_N \Vert_{L^2(\Gamma)},
\end{equation*}
i.e., 
\begin{equation} \label{eqpenun}
\Vert u_N \Vert_{L^2(\Gamma)} \leq \tilde{C}  \eta^{-1} \Vert \bbf \Vert_{L^2(\Gamma)}.
\end{equation}
Testing problem \eqref{projectedcontform0} and problem \eqref{contform} with $\bv_T = \bu_T - \bu_T^\ast$ results in
$
a(\bu_T^\ast, \bv_T) - a(\bu, \bv_T) = 0$, and thus 
$
a(\bv_T, \bv_T) = - a(u_N \bn,\bv_T)$.
Using Korn's inequality (Lemma~\ref{lemmakorn}) and continuity  of $a(\cdot,\cdot)$ we get
\begin{equation*}
\Vert \bu_T - \bu_T^\ast \Vert_{H^1(\Gamma)}^2 \leq \hat{C}  \Vert u_N \Vert_{L^2(\Gamma)}\Vert \bu_T - \bu_T^\ast \Vert_{H^1(\Gamma)}.
\end{equation*}
Combining this with \eqref{eqpenun} proves the result \eqref{penerror}.
\end{proof}
 
To obtain a consistent variant of this formulation we introduce the bilinear form $a_T(\cdot,\cdot)$ in which only the tangential components of the arguments play a role:
\begin{equation} \label{bilinearformpenalty} 
a_T(\bu, \bv) := a(\bP\bu,\bP\bv)=a(\bu_T,\bv_T). 
\end{equation}
The corresponding penalty formulation is: for a given $\bbf \in L^2(\Gamma)^3$ with $\bbf \cdot \bn = 0$ determine $\textbf{u} \in \bV_*$, such that
\begin{equation} \tag{P2}\label{projectedcontform} 
a_T(\bu, \bv) + k(\bu, \bv) = ( \textbf{f}, \textbf{v} )_{L^2(\Gamma)} \quad \text{for all}~~\textbf{v} \in \bV_*.
\end{equation}
This formulation is indeed consistent:
\begin{lemma}
Problem \eqref{projectedcontform} is well posed. For the unique solution $\tilde{\bu} = \tilde{\bu}_T \in \bV_*$ of this problem we have $\tilde{\bu}_T = \bu_T^*$.
\end{lemma}
\begin{proof} Define $A(\bu,\bv)=a_T(\bu, \bv) + k(\bu, \bv)$. We have $|A(\bu,\bv)|\leq c\Vert \bu \Vert_{V_\ast}\Vert \bv \Vert_{V_\ast}$ for all $\bu,\bv \in V_\ast$, and
using Lemma \ref{lemmakorn} it follows that  there is a constant $c_0 > 0$ such that 
\begin{equation*}
 \Vert \bu \Vert_{V_*}^2 = \Vert \bu_T \Vert_{H^1(\Gamma)}^2 + \Vert u_N \Vert_{L^2(\Gamma)}^2  \leq  c_0 A(\bu,\bu) \quad \text{for all}~~ \bu \in V_\ast.
\end{equation*}
Therefore problem \eqref{projectedcontform} is well posed. For the unique solution $\bu_T^* \in \bV_T$ of problem \eqref{contform} we have
\begin{equation*}
A(\bu_T^\ast, \bv)=
a_T(\bu_T^*, \bv) + k(\bu_T^*,\bv) = a(\bu_T^*, \bv_T) = ( \textbf{f}, \textbf{v} )_{L^2(\Gamma)} \quad \text{for all}~~\bv \in \bV_*.
\end{equation*}
Hence, $\bu_T^*$ solves problem \eqref{projectedcontform}. 
\end{proof}

The third formulation that we  consider uses a Lagrange multiplier to ensure that the solution is tangential to $\Gamma$. We use the bilinear form $a(\cdot,\cdot)$ as in \eqref{defbla} and $b(\bu,\mu):=(\bu \cdot \bn, \mu)_{L^2(\Gamma)}$, $\bu \in V_\ast, \mu \in L^2(\Gamma)$. For a given $\bg \in L^2(\Gamma)^3$, which is not necessarily tangential, we introduce the following saddle point problem: determine $(\bu,\lambda) \in \bV_\ast \times L^2(\Gamma)$ such that
\begin{equation} \tag{L}\label{contformlagrange}\begin{aligned}
a(\bu,\bv) + b(\bv,\mu) &=(\mathbf{g},\bv)_{L^2(\Gamma)} &\quad &\text{for all }\bv \in \bV_\ast, \\
 b(\bu,\mu) & = 0 &\qquad &\text{for all }\mu \in L^2(\Gamma).
\end{aligned}
\end{equation}

Well-posedness of this saddle point problem is derived in the following theorem (see \cite{grossvectorlaplace}).
\begin{theorem}
The problem \eqref{contformlagrange} is well-posed. Its unique solution $(\hat{\bu},\lambda) \in \bV_\ast \times L^2(\Gamma)$ has the following properties:
\begin{align}
1.\quad&\hat{\bu} \cdot \bn  =0,\\
2.\quad&\hat{\bu}_T =\bu_T^*,~~  \text{where } \bu_T^* \text{ is the unique solution of \eqref{contform} with } \mathbf{f}:=\bg_T=\bP \bg, \label{hhj} \\
3.\quad&\lambda  = g_N-  \tr \big(E(\hat{\bu}_T) \bH\big),~~ \text{for}~\bg=\bg_T+g_N\bn. \label{charlambda}
\end{align}
\end{theorem}
Summarizing, for the given vector Laplace problem \eqref{contform} we have two  alternative consistent formulations, namely \eqref{projectedcontform} (penalty approach) and \eqref{contformlagrange} (Lagrange multiplier), and one inconsistent formulation \eqref{projectedcontform0} (penalty approach). In the following sections we present a detailed analysis of  finite element methods based on these different formulations.  

\section{Parametric finite element space} \label{sectParametric}
For the discretization of the different variational problems \eqref{projectedcontform0}, \eqref{projectedcontform} and \eqref{contformlagrange}  we use a parametric trace finite element approach as in \cite{grande2017higher}. In this section we define the  finite element space used in this method and summarize certain properties, known from the literature, that we need in the error analysis of the finite element methods.

Let $\{ \mathcal{T}_h \}_{h>0}$ be a family of shape regular tetrahedral triangulations of $\Omega$. For simplicity, in the analysis of the method, we assume $\{\T_h\}_{h >0}$ to be quasi-uniform. By $V_h^k$ we denote the standard finite element space of continuous piecewise polynomials of degree $k$. The nodal interpolation operator in $V_h^k$ is denoted by $I^k$. As input for the parametric mapping we need an approximation of $\phi$. We consider geometry approximations whose order of approximation may differ from the order of the polynomials used in the finite element. In other words, the spaces introduced below are not necessarily \emph{iso}parametric.  Let $k_g$ be the geometry approximation order, i.e., the construction of the geometry approximation will be based on a level set function approximation  $\phi_h \in V_h^{k_g}$ that satisfies the error estimate
\begin{equation*}
\max_{T \in \mathcal{T}_h} \vert \phi_h - \phi \vert_{W^{l,\infty}(T\cap U_\delta)} \lesssim h^{k_g+1-l}, \quad 0\leq l \leq k_g+1.  
\end{equation*}
Here, $\vert \cdot \vert_{W^{l,\infty}(T\cap U_\delta)}$ denotes the usual semi-norm on the Sobolev space $W^{l,\infty}(T\cap U_\delta)$ and the constant used in $\lesssim$ depends on $\phi$ but is independent of $h$. The zero level set of the finite element function $\phi_h$ \emph{implicitly} characterizes an approximation of the interface, which, however, is hard to compute for $k_g \geq 2$.  With the piecewise \emph{linear} nodal interpolation of $\phi_h$, which is denoted by $\hat{\phi}_h = I^1\phi_h$, we define the low order geometry approximation:
\begin{equation*}
\Gamma^{\text{lin}} := \{ x \in \Omega \mid \hat{\phi}_h (x) = 0\}.
\end{equation*}
The  tetrahedra $T \in \cT_h$ that have a nonzero intersection with $\Gamma^{\text{lin}}$ are collected in the set denoted by $\T_h^\Gamma$. The domain formed by all tetrahedra in $\T_h^\Gamma$ is denoted by $\OGamma:= \{ x \in T \mid T \in \T_h^\Gamma \}$. Let $\Theta_h^{k_g}\in  (V_h^{k_g})_{\OGamma}^3$ be the mesh transformation of order $k_g$ as defined in \cite{grande2017higher}, cf. Remark~\ref{transfo}.
\begin{remark} \label{transfo} \rm We outline the key idea of the mesh transformation $\Theta_h^{k_g}$. For a detailed description we refer to \cite{grande2017higher}, \cite{lehrenfeld2016high} and \cite{lehrenfeld2017analysis}. 
There exists a unique $ \tilde d \colon \Omega_h^\Gamma \to \mathbb{R}$ defined as follows: $ \tilde d(x)$ is the in absolute value smallest number such that
\begin{equation*}
\phi(x +  \tilde d(x) \nabla \phi(x)) = \hat{\phi}_h(x) \qquad \text{for } x \in \Omega_h^\Gamma.
\end{equation*}
Based on $\tilde d$ we define the mapping
\begin{equation*}
\Psi(x) := x + \tilde d(x) \nabla \phi(x), \qquad x \in \Omega_h^\Gamma,
\end{equation*}
which has the property $\Psi(\Gamma^{\text{lin}}) = \Gamma$. 
To avoid comptations with $\phi$ (which even may not be available) we use a similar construction with $\phi$ replaced by its (finite element) approximation $\phi_h$.
The resulting mapping $\Psi_h$ is not necessarily a finite element function. The  mesh transformation $\Theta_h^{k_g}$ is obtained by projection of $\Psi_h$ into the finite element space $(V_h^{k_g})_{\OGamma}^3$.
\end{remark}

The approximation of $\Gamma$ is defined as
\begin{equation*}
\Gamma_{h}^{k_g} := \Theta_h^{k_g}(\Gamma^{\text{lin}}) = \left\lbrace x \mid \hat{\phi}_h((\Theta_{h}^{k_g})^{-1}(x)) = 0 \right\rbrace.
\end{equation*}
We denote the transformed cut mesh domain by $\Omega^{\Gamma}_\Theta := \Theta_h^{k_g}(\Omega^\Gamma_h)$. We assume that $h$ is small enough such that $\Omega^{\Gamma}_\Theta \subset U_\delta$ holds. We apply to $V_h^{k}$ the transformation $\Theta_h^{k_g}$ resulting in the parametric spaces
\begin{equation*} 
V_{h,\Theta}^{k,k_g} := \left\lbrace v_h \circ (\Theta_h^{k_g})^{-1} \mid v_h \in (V_h^{k})_{\Omega^{\Gamma}_h} \right\rbrace, \quad
\bV_{h,\Theta}^{k,k_g} := (V_{h,\Theta}^{k,k_g})^3.
\end{equation*}
Note that $k_g$ denotes the degree of the polynomials used in the parametric mapping $\Theta_h^{k_g}$, and $k$ the degree of the polynomials used in the finite element space. 
To simplify the notation we delete the superscript $k_g$ and write  $V_{h,\Theta}^{k} = V_{h,\Theta}^{k,k_g}$, $\bV_{h,\Theta}^{k} = \bV_{h,\Theta}^{k,k_g}$, $\Theta_h = \Theta_h^{k_g}$ and $\Gamma_{h} = \Gamma_{h}^{k_g}$.  Here and further in the paper we write $x \lesssim y$ to state that there exists a constant $c>0$, which is independent of the mesh parameter $h$ and the position of $\Gamma$ and $\Gamma_h$ in the background mesh, such that the inequality $x \leq cy$ holds. Similarly for $x\gtrsim y$, and $x \sim y$  means that both $x\lesssim y$ and $x\gtrsim y$ hold.

We recall some known approximation results from the literature \cite{grande2017higher}.  The parametric interpolation $I_{\Theta}^{k} \colon C(\Omega_{\Theta}^{\Gamma}) \to V_{h,\Theta}^{k}$ is defined by $(I_{\Theta}^{k} v)\circ \Theta_h = I^{k}(v\circ \Theta_h)$. We have the following optimal interpolation error bound for $0\leq l \leq {k}+1$:
\begin{equation} \label{eqinterpolationerror}
\Vert v - I_{\Theta}^{k} v \Vert_{H^l(\Theta_h(T))} \lesssim h^{k+1-l} \Vert v  \Vert_{H^{k+1}(\Theta_h(T))} ~~~\text{for all} ~ v \in H^{k+1}(\Theta_h(T)), T \in \mathcal{T}_h.
\end{equation}
We also need the following trace estimate (\cite{Hansbo02}):
\begin{equation} \label{eqtraceestimate}
\Vert v \Vert_{L^2(\Gamma_T)}^2 \lesssim h^{-1} \Vert  v \Vert_{L^2(\Theta_h(T))}^2 + h \Vert  \nabla v \Vert_{L^2(\Theta_h(T))}^2 ~~~ \text{for} ~ v \in H^1(\Theta_h(T)),
\end{equation}
with $\Gamma_T := \Gamma_h \cap \Theta_h(T)$. The Sobolev norms on $\Omega_{\Theta}^{\Gamma}$ of the normal extension $u^e$  can be estimated by the corresponding norms on $\Gamma$ (\cite{reusken2015analysis}):
\begin{equation} \label{lemmasobolevnormsneighborhood}
\Vert D^\mu u^e \Vert_{L^2(\Omega_{\Theta}^{\Gamma})} \lesssim h^{\frac{1}{2}} \Vert u \Vert_{H^m(\Gamma)} \qquad \text{for all} ~ u \in H^m(\Gamma), \, \vert \mu \vert \leq m.
\end{equation}

\begin{lemma} \label{lemmascalarapproximationerror}
For the space $V_{h,\Theta}^{k}$ we have the approximation error estimate
\begin{equation*} \begin{split}
&\min_{v_h \in V_{h,\Theta}^{k}} \left( \Vert v^e - v_h \Vert_{L^2(\Gamma_h)} + h \Vert \nabla( v^e - v_h) \Vert_{L^2(\Gamma_h)} \right) \\
&\leq \Vert v^e - I_{\Theta}^{k} v^e \Vert_{L^2(\Gamma_h)} + h \Vert \nabla( v^e - I_{\Theta}^{k} v^e) \Vert_{L^2(\Gamma_h)} \lesssim h^{k+1} \Vert v \Vert_{H^{k+1}(\Gamma)}
\end{split}
\end{equation*}
for all $v \in H^{k+1}(\Gamma)$.
\end{lemma}
\begin{proof} The proof uses standard arguments, based on \eqref{eqinterpolationerror}, \eqref{eqtraceestimate} and \eqref{lemmascalarapproximationerror}, cf. \cite{grande2017higher}.
\end{proof}

The following lemma, taken from \cite{grande2017higher}, gives an approximation error for the easy to compute normal approximation $\bn_h$, which is used in the methods introduced below.
\begin{lemma} \label{lemmanormals}
For $x \in T \in \mathcal{T}^\Gamma_h$ define
\begin{equation*}
\bn_{\textrm{lin}} = \bn_{\textrm{lin}}(T) := \frac{\nabla \hat{\phi}_h(x)}{\Vert \nabla \hat{\phi}_h(x)\Vert_2} = \frac{\nabla \hat{\phi}_{h|T}}{\Vert \nabla \hat{\phi}_{h|T}\Vert_2}, \quad \bn_h (\Theta(x)) := \frac{D\Theta_h(x)^{-T} \bn_{\textrm{lin}}}{\Vert D\Theta_h(x)^{-T} \bn_{\textrm{lin}} \Vert_2}.
\end{equation*}
Let $\bn_{\Gamma_h}(x)$, $x \in \Gamma_h$ a.e., be the unit normal on $\Gamma_h$ (in the direction of $\phi_h >0$). The following holds:
\begin{equation*} \begin{split}
\Vert \bn_h - \bn \Vert_{L^\infty(\Omega_\Theta^\Gamma)} \lesssim h^{k_g}, \qquad
\Vert \bn_{\Gamma_h} - \bn \Vert_{L^\infty(\Gamma_h)} \lesssim h^{k_g}.
\end{split}
\end{equation*}
\end{lemma}

Similar to the extension of a function $u$ defined on $\Gamma$ to $u^e$ defined on $U_\delta$ we define the lifting $u^l$ of a function $u$ defined   on $\Gamma_h$ by 
\begin{equation*}
\begin{cases}
u^l(p(x)) = u(x) &\text{ for } x \in \Gamma_h, \\
u^l(x) = u^l(p(x)) &\text{ for } x \in U_\delta.
\end{cases}
\end{equation*} 
A norm on $H^1(\Gamma_h)^3$ is defined using the component-wise lifting by
\begin{equation*}
\Vert \bu \Vert_{H^1(\Gamma_h)}^2 := \int_{\Gamma_h} \Vert \bu(s) \Vert_2^2 + \Vert \nabla \bu^l(s) \bP_h(s) \Vert_2^2 \, ds
\end{equation*}
with $\bP_h = \bI - \bn_h \bn_h^T$. 
In \eqref{eqdefH1} the term with $\nabla \bu^e$ corresponds to tangential gradients of all components, cf. \eqref{defH1}.
 The lifting used in the definition of the $H^1(\Gamma_h)$-norm is constant along the normal to $\Gamma$ (not $\Gamma_h$). Therefore,  to eliminate the part of the (componentwise) gradient which is normal to $\Gamma_h$ one uses the projection $\bP_h$. We also introduce the  following  spaces
\begin{equation*} \begin{split}
V_{reg,h} := \left\lbrace v \in H^1(\Omega_{\Theta}^{\Gamma}) \mid \text{tr} |_{\Gamma_h} v \in H^1(\Gamma_h) \right\rbrace \supset V_{h,\Theta}^{k}, \quad
\bV_{reg,h} := \big(V_{reg,h}\big)^3.
\end{split}
\end{equation*}
\section{Parametric trace finite element methods} \label{sectFEmethods}
In this section we introduce three parametric trace finite element methods. These are obtained by applying a Galerkin method (modulo a geometry error due to $\Gamma_h \approx \Gamma$) to the three formulations \eqref{projectedcontform0}, \eqref{projectedcontform} and \eqref{contformlagrange} with the parametric finite element space $\bV_{h,\Theta}^{k}$. 

We introduce further notation. In particular,   discrete variants of the bilinear forms $a(\cdot,\cdot)$, $a_T(\cdot,\cdot)$ and the penalty bilinear form $k(\cdot,\cdot)$ introduced above. Since we use a trace FEM, we need a stabilization that eliminates instabilities caused by the small cuts. For this we use the so-called ``normal derivative volume stabilization'' \cite{grande2017higher}  ($s_h(\cdot,\cdot)$ and $\tilde{s}_h(\cdot,\cdot)$ below):
\begin{align*}
 \gradGh \bu(x) &:= \bP_h(x) \nabla \bu(x) \bP_h(x),\quad x \in \Gamma_h, \\ 
  E_h(\bu)&  := \frac12 \big(\gradGh \bu + \gradGh^T \bu\big), \quad E_{T,h}(\bu):=E_h(\bu) - u_N \bH_h, \\
 a_h(\bu,\bv) &:= \int_{\Gamma_h} E_h(\bu):E_h(\bv)\, ds_h + \int_{\Gamma_h} \bu_h \cdot \bv_h \, ds_h, \\
 a_{T,h}(\bu,\bv) &:= \int_{\Gamma_h} E_{T,h}(\bu):E_{T,h}(\bv)\, ds_h + \int_{\Gamma_h} \bP_h\bu_h \cdot \bP_h\bv_h \, ds_h,\\
 k_h(\bu,\bv)&:= \eta \int_{\Gamma_h} (\bu \cdot \tilde{\bn}_h) (\bv \cdot \tilde{\bn}_h)  \, ds_h,\qquad
  s_h(\bu,\bv) := \rho \int_{\Omega_{\Theta}^{\Gamma}} (\nabla \bu \bn_h) \cdot (\nabla \bv \bn_h)  \, dx, \\
  b_h(\bu,\mu)& := (\bu\cdot \bn_h,\mu)_{L^2(\Gamma_h)} + \tilde{s}_h(\bu,\mu), ~ \tilde{s}_h(\bu,\mu) := \tilde \rho  \int_{\Omega_{\Theta}^{\Gamma}} (\bn_h^T \nabla \bu \bn_h) (\bn_h \cdot \nabla \mu) \, dx.
\end{align*}
All these bilinear forms are well-defined for $\bu,\bv \in \bV_{reg,h}$, $\mu  \in V_{reg,h}$.
The normal vector $\tilde{\bn}_h$, used in the penalty term $k_h(\cdot, \cdot)$, and the curvature tensor $\bH_h$ are approximations of the exact normal and the exact Weingarten mapping, respectively. The choice of the stabilization parameters $\rho$, $\tilde \rho$ is discussed below. 
\begin{remark} \rm
We use $E_{T,h}(\bu):=E_h(\bu) - u_N \bH_h$ instead of $E_{T,h}(\bu)=E_h(\bP_h\bu)$ because the latter requires (tangential) differentiation of $\bP_h$, which has certain disadvantages.
The reason that we introduce yet another normal approximation $\tilde{\bn}_h$ is the following. In the analysis below we will see that in order to achieve optimal order estimates  we need the normal $\tilde{\bn}_h$ used in the penalty term to be an approximation of at least one order higher than the normal approximation $\bn_h$. 
An approximation $\bH_h$ of the Weingarten map can be easily obtained, e.g.,  by taking $\bH_h = \nabla(I_{\Theta}^{k_g}(\bn_h))$. 
The stabilization with  $s_h(\cdot,\cdot)$ used in the variational penalty formulations below guarantees that the stiffness matrix has a spectral condition number $\sim h^{-2}$, independent of how the interface cuts the outer triangulation.
\end{remark}

To quantify the error in the approximations $\tilde \bn_h \approx  \bn$, $\bH_h \approx \bH$, we introduce one further order parameter $k_p$ (besides $k$ and $k_g$) and assume:
\begin{align}
\Vert \bn - \tilde{\bn}_h \Vert_{L^{\infty}(\Gamma_h)} &\lesssim h^{k_p}, \quad k_p \geq k_g,\\
\Vert \bH - \bH_h \Vert_{L^{\infty}(\Gamma_h)} &\lesssim h^{k_{g}-1}. \label{eqapproxH}
\end{align}
We now introduce discrete versions of the formulations \eqref{projectedcontform0}, \eqref{projectedcontform} and \eqref{contformlagrange}. For these we  need a suitable extension of the data $\bbf$ to $\Gamma_h$, which is denoted by $\bbf_h$. 
\\
\emph{Discrete inconsistent penalty formulation.} This problem is as follows: determine $\bu_h \in  \bV_{h,\Theta}^{k}$ such that for all $\bv_h \in \bV_{h,\Theta}^{k}$
\begin{equation} 
A_h^{P_1}(\bu_h,\bv_h): =  a_h(\bu_h,\bv_h) + s_h(\bu_h,\bv_h) + k_h(\bu_h, \bv_h) = (\bbf_h, \bv_h)_{L^2(\Gamma_h)}. \tag{P1h}\label{discretepenaltyform1}
\end{equation}
\emph{Discrete consistent penalty formulation.} This problem is as follows: determine
 $\bu_h \in  \bV_{h,\Theta}^{k}$ such that for all $\bv_h \in \bV_{h,\Theta}^{k}$
\begin{equation} 
A_h^{P_2}(\bu_h,\bv_h): = a_{T,h}(\bu_h,\bv_h) + s_h(\bu_h,\bv_h) + k_h(\bu_h, \bv_h)=(\bbf_h, \bv_h)_{L^2(\Gamma_h)}. \tag{P2h}\label{discretepenaltyform2}
\end{equation}
\emph{Discrete  Lagrange multiplier formulation.} This problem is as follows:  determine $(\bu_h, \lambda_h) \in \bV_{h,\Theta}^{k} \times V_{h,\Theta}^{k_l}$ such that
\begin{equation} \tag{Lh} \label{discretelagrangeform}
\begin{aligned}
A_h^L(\bu_h,\bv_h) + b_h(\bv_h,\lambda_h) & =(\mathbf{f}_h,\bv_h)_{L^2(\Gamma)} &\quad &\text{for all } \bv_h \in \bV_{h,\Theta}^{k} \\ 
b_h(\bu_h,\mu_h) & = 0 &\quad &\text{for all }\mu_h \in V_{h,\Theta}^{k_l}
\end{aligned}
\end{equation}
with 
$
 A_h^L(\bu,\bv) := a_h(\bu,\bv) + s_h(\bu,\bv)$.
%
\begin{remark} \rm
Problem \eqref{discretelagrangeform} uses a Lagrange multiplier approach to enforce the tangential condition weakly. This formulation is consistent without using additional tangential projections in $a_h(\cdot,\cdot)$ and avoids the approximation of the Weingarten map $\bH$. An obvious drawback  of this formulation is that the resulting linear systems can be significantly larger than the ones in the penalty formulation. In addition to the stabilization $s_h(\cdot, \cdot)$ we  use a "normal derivative volume" stabilization for the Lagrange multiplier term as well. Different from  $s_h(\cdot, \cdot)$ this stabilization $\tilde s_h(\cdot, \cdot)$ is essential for the well-posedness of this formulation, cf. section~\ref{sectwellposed}.
\end{remark}

\section{Error analysis of the parametric TraceFEM} \label{sectAnalysis}

In this section we present an error analysis of the TraceFEMs  \eqref{discretepenaltyform1}, \eqref{discretepenaltyform2} and \eqref{discretelagrangeform}. We first address the choice of the stabilization parameters  $\rho$, $\tilde \rho$. From the analysis in \cite{grande2017higher} it is known that for optimal error bounds  one must  restrict $\rho$ to the range $h \lesssim \rho \lesssim h^{-1}$. A more detailed analysis (that we do not present here) has shown that there are no significant gains if one chooses for $\tilde \rho$ a different scaling w.r.t $h$ as for $\rho$. Therefore, to simplify the presentation, in the remainder we restrict the stabilization parameters to 
\begin{equation} \label{choicerho}
 h \lesssim \rho =\tilde \rho  \lesssim h^{-1}.
\end{equation}

\subsection{Well-posedness of discretizations} \label{sectwellposed}

We start with some basic results concerning the bilinear forms.
We use the following natural norms
\begin{equation*} \begin{split}
\Vert \bu \Vert_{A_h^{P_i}}^2 &:= A_h^{P_i}(\bu, \bu), ~~i=1,2, \quad 
\Vert \bu \Vert_{A_h^L}^2 := A_h^L(\bu, \bu), \\
\Vert \mu \Vert_{M}^2 &:= \Vert \mu \Vert_{L^2(\Gamma_h)}^2 + \rho \Vert \bn_h \cdot \nabla \mu \Vert_{L^2(\Omega_{\Theta}^{\Gamma})}^2 .
\end{split}
\end{equation*}
Before we analyze continuity and ellipticity of the bilinear forms we recall a lemma which shows that for finite element functions the $L^2$-norm in the neighborhood $\Omega_{\Theta}^{\Gamma}$ can be controlled by the $L^2$-norm on $\Gamma_h$ and the $L^2$-norm of the normal derivative  on $\Omega_{\Theta}^{\Gamma}$. 
\begin{lemma} \label{lemmastabilizationinequality}
For all $k \in \mathbb{N}$, $k\geq 1$, the following inequality holds:
\begin{equation*}
\Vert v_h \Vert_{L^2(\Omega_{\Theta}^{\Gamma})}^2 \lesssim h\Vert v_h \Vert_{L^2(\Gamma_h)}^2 + h^2\Vert \bn_h \cdot \nabla v_h \Vert_{L^2(\Omega_{\Theta}^{\Gamma})}^2   ~~~~\text{for all} ~v_h \in V_{h,\Theta}^{k}.
\end{equation*}
\end{lemma}
\begin{proof}
 In \cite[Lemma 7.8]{grande2017higher} a proof of this result for the isoparametric case (i.e. $k =k_g$) is given. This proof also applies to the case $k\neq k_g$.
\end{proof}

We formulate a few corollaries that are   useful in the remainder. Using the trace inequality \eqref{eqtraceestimate} and a standard finite element inverse inequality one obtains
\begin{align}
 \Vert v_h \Vert_{L^2(\Omega_{\Theta}^{\Gamma})}^2 &\sim h\Vert v_h \Vert_{L^2(\Gamma_h)}^2 + h^2\Vert \bn_h \cdot \nabla v_h \Vert_{L^2(\Omega_{\Theta}^{\Gamma})}^2   ~~~\text{for all} ~v_h \in V_{h,\Theta}^{k}, \label{HH1} \\
 \Vert \bv_h \Vert_{L^2(\Omega_{\Theta}^{\Gamma})}^2 &\sim h\Vert \bv_h \Vert_{L^2(\Gamma_h)}^2 + h^2\Vert \nabla \bv_h \bn_h \Vert_{L^2(\Omega_{\Theta}^{\Gamma})}^2   ~~~\text{for all} ~\bv_h \in \bV_{h,\Theta}^{k}.\label{HH2}
\end{align}
The result in \eqref{HH2} is obtained by componentwise application of \eqref{HH1}. A further direct consequence of Lemma~\ref{lemmastabilizationinequality} is  (we use \eqref{choicerho}):
\begin{equation} \label{HH3}
 \Vert v_h \Vert_{L^2(\Omega_{\Theta}^{\Gamma})} \lesssim h^\frac12 \|v_h\|_M~~~\text{for all} ~v_h \in V_{h,\Theta}^{k}.
\end{equation}
Using \eqref{eqtraceestimate} and \eqref{HH1} we also obtain the inverse inequality
\begin{equation} \label{HH4}
 \|\nabla v_h \|_{L^2(\Gamma_h)} \lesssim h^{-1} \|v_h\|_{L^2(\Gamma_h)} + h^{-\frac12} \|\bn_h \cdot \nabla v_h\|_{L^2(\Omega_{\Theta}^{\Gamma})} \lesssim h^{-1}\|v_h\|_M,~ ~v_h \in V_{h,\Theta}^{k},
\end{equation}
and the vector analogon
\begin{equation} \label{HH4a}
 \|\nabla \bv_h \|_{L^2(\Gamma_h)} \lesssim h^{-1} \|\bv_h\|_{L^2(\Gamma_h)} + h^{-\frac12} \| \nabla \bv_h\bn_h\|_{L^2(\Omega_{\Theta}^{\Gamma})},~~
\bv_h \in \bV_{h,\Theta}^{k}.
\end{equation}
\begin{lemma} \label{Ahscalarproduct}
The following inequalities hold:
\begin{align*} 
A_h^{P_i}(\bu,\bv) &\leq \Vert \bu \Vert_{A_h^{P_i}} \Vert \bv \Vert_{A_h^{P_i}} ~~~ \text{for all}~ \bu, \bv \in \bV_{reg,h},~~i=1,2.\\ 
A_h^L(\bu,\bv) &\leq \Vert \bu \Vert_{A_h^L} \Vert \bv \Vert_{A_h^L} ~~~ \text{for all}~ \bu, \bv \in \bV_{reg,h}, \\
b_h(\bu,\mu) &\leq \Vert \bu \Vert_{A_h^L} \Vert \mu \Vert_{M} ~~~ \text{for all}~ \bu \in \bV_{reg,h}, \mu \in V_{reg,h},\\
A_h^{P_i}(\bu_h, \bu_h) &\gtrsim h^{-1} \Vert \bu_h \Vert_{L^2(\Omega_{\Theta}^{\Gamma})}^2 ~~~ \text{for all} ~ \bu_h \in \bV_{h,\Theta}^{k}, ~i=1,2,\\
A_h^L(\bu_h, \bu_h) &\gtrsim h^{-1} \Vert \bu_h \Vert_{L^2(\Omega_{\Theta}^{\Gamma})}^2 ~~~ \text{for all} ~ \bu_h \in \bV_{h,\Theta}^{k}.
\end{align*}
\end{lemma}
\begin{proof}
The first three  estimates follow directly from the Cauchy-Schwarz inequality. To show the other two we use \eqref{HH2}:
\begin{equation*} \begin{split}
A_h^{P_i}(\bu_h, \bu_h) & \geq    \Vert \bu_h \Vert_{L^2(\Gamma_h)}^2 + \rho \Vert \nabla \bu_h \bn_h \Vert_{L^2(\Omega_\Theta^\Gamma)}^2 
\gtrsim h^{-1} \Vert \bu_h \Vert_{L^2(\Omega_\Theta^{\Gamma})}^2 \\
A_h^L(\bu_h, \bu_h) &\geq   \Vert \bu_h \Vert_{L^2(\Gamma_h)}^2 + \rho \Vert \nabla \bu_h \bn_h \Vert_{L^2(\Omega_\Theta^\Gamma)}^2 \gtrsim h^{-1} \Vert \bu_h \Vert_{L^2(\Omega_\Theta^{\Gamma})}^2.
\end{split}
\end{equation*}
\end{proof}

From Lemma \ref{Ahscalarproduct} it follows that the discrete penalty problems \eqref{discretepenaltyform1} and \eqref{discretepenaltyform2} have unique solutions. For well-posedness of the discrete Lagrange multiplier formulation \eqref{discretelagrangeform} we need a discrete inf-sup estimate. This we will now derive. We outline the idea of the analysis. For the bilinear form $b(\bu,\mu)=(\bu\cdot\bn,\mu)_{L^2(\Gamma)}$ used in the continuous problem \eqref{contformlagrange} we have, for arbitray $\mu \in L^2(\Gamma)$ and with $\hat \bu:= \mu \bn$ that $b(\hat \bu,\mu)=\|\mu\|_{L^2(\Gamma)}^2$. Furthermore, $a(\hat \bu,\hat \bu)= \|\hat u_N \bH\|_{L^2(\Gamma)}^2 +\|\hat \bu\|_{L^2(\Gamma)}^2 \leq c \|\mu\|_{L^2(\Gamma)}^2$ holds. From this the inf-sup property of $b(\cdot,\cdot)$ for the continuous problem can easily be concluded. For deriving a discrete inf-sup result we combine this approach with perturbation arguments. In Lemma~\ref{lemmainfsup1}  we analyze the perturbation $b_h(\hat \bu,\mu)- b(\hat\bu,\mu)$, and in Lemma~\ref{lemmainfsup2}  we derive the discrete analogon of $a(\hat \bu,\hat \bu) \leq c \|\mu\|_{L^2(\Gamma)}^2$. Combining these results we obtain the discrete inf-sup property in Lemma~\ref{lemdiscreteinfsup}.

\begin{lemma} \label{lemmainfsup1}
For $h$ small enough the following inequality holds:
\begin{equation*}
b_h(\mu_h \bn, \mu_h) \gtrsim \Vert \mu_h \Vert_{M}^2 \qquad \text{for all } \mu_h \in V_{h,\Theta}^{m}.
\end{equation*}
\end{lemma}
\begin{proof}
Using Lemma \ref{lemmanormals} we get
\begin{equation*}
2-2\bn \cdot\bn_h \leq \Vert \bn - \bn_h \Vert_{L^\infty(\Gamma_h)}^2 \lesssim h^{2k_g} \qquad \text{a.e. on } \Gamma_h.
\end{equation*}
Hence, there exists a constant $c>0$ with
\begin{equation} \label{eqnnh}
1- ch^{2k_g} \lesssim \bn \cdot \bn_h .
\end{equation}
Take $\mu_h \in V_{h,\Theta}^{m}$. From the definition of $b_h(\cdot, \cdot)$ we obtain
\begin{equation}  \label{eqinfsup1}
b_h(\mu_h \bn, \mu_h) =\underbrace{ (\mu_h \bn \cdot \bn_h, \mu_h)_{L^2(\Gamma_h)}}_{(I)} +   \underbrace{\rho \int_{\Omega_{\Theta}^{\Gamma}} (\bn_h^T \nabla (\mu_h \bn) \bn_h) (\bn_h \cdot \nabla \mu_h) \, dx}_{(II)}.
\end{equation}
Using inequality \eqref{eqnnh} the term (I) can be estimated by
\begin{equation}  \label{eqinfsup2}
(\mu_h \bn \cdot \bn_h, \mu_h)_{L^2(\Gamma_h)} \gtrsim
(1-ch^{2k_g}) \Vert \mu_h \Vert_{L^2(\Gamma_h)}^2,
\end{equation}
and term (II) by
\begin{equation}  \label{eqinfsup3} \begin{split}
& \rho \int_{\Omega_{\Theta}^{\Gamma}} (\bn_h^T \nabla (\mu_h \bn) \bn_h) (\bn_h \cdot \nabla \mu_h) \, dx \\
&= \rho \int_{\Omega_{\Theta}^{\Gamma}} (\bn_h \cdot \nabla \mu_h)( \bn \cdot \bn_h) (\bn_h \cdot \nabla \mu_h) \, dx + \rho \int_{\Omega_{\Theta}^{\Gamma}} \mu_h(\bn_h^T \nabla \bn \bn_h) (\bn_h \cdot \nabla \mu_h) \, dx \\
&\gtrsim (1-ch^{2k_g}) \rho \Vert\bn_h \cdot  \nabla \mu_h \Vert_{L^2(\Omega^\Gamma_\Theta)}^2 + \rho \int_{\Omega_{\Theta}^{\Gamma}} \mu_h(\bn_h^T \nabla \bn \bn_h) (\bn_h \cdot \nabla \mu_h) \, dx.
\end{split}
\end{equation}
Since $\nabla \bn \bn =0$ and $\bn^T \nabla \bn = 0$ we get for the last term on the right hand side
\begin{align*}
&\rho \int_{\Omega_{\Theta}^{\Gamma}} \mu_h(\bn_h^T \nabla \bn \bn_h) (\bn_h \cdot \nabla \mu_h) \, dx \\
&=\rho \int_{\Omega_{\Theta}^{\Gamma}} \mu_h\big((\bn_h^T - \bn^T) \nabla \bn (\bn_h-\bn)\big) (\bn_h \cdot \nabla \mu_h) \, dx \\
&\geq - \rho \Vert \nabla \bn \Vert_{L^\infty(\Omega^\Gamma_\Theta)} \Vert \bn_h - \bn \Vert_{L^\infty(\Omega^\Gamma_\Theta)}^2 \Vert \mu_h \Vert_{L^2(\Omega^\Gamma_\Theta)} \Vert \bn_h \cdot \nabla \mu_h \Vert_{L^2(\Omega^\Gamma_\Theta)} \\
&\gtrsim - h^{2k_g} \Vert \mu_h \Vert_{L^2(\Omega^\Gamma_\Theta)} \rho \Vert \bn_h \cdot \nabla \mu_h \Vert_{L^2(\Omega^\Gamma_\Theta)} \\
&\overset{ \eqref{HH3} }{\gtrsim} - h^{2k_g} \Vert \mu_h \Vert_{M} \rho^\frac12\Vert \bn_h \cdot \nabla \mu_h \Vert_{L^2(\Omega^\Gamma_\Theta)}  \gtrsim - h^{2k_g} \Vert \mu_h \Vert_{M}^2.
\end{align*}
Combined with \eqref{eqinfsup1}, \eqref{eqinfsup2} and \eqref{eqinfsup3} we obtain
\begin{equation*}
b_h(\mu_h \bn, \mu_h)
\gtrsim \left( 1- \tilde{c} h^{2k_g} \right) \Vert \mu_h \Vert_{M}^2
\gtrsim \Vert \mu_h \Vert_{M}^2, 
\end{equation*}
provided $h$ is sufficiently small.
\end{proof}

\begin{lemma} \label{lemmainfsup2}
 Take $\mu_h \in V_{h,\Theta}^{m}$ and define $\bv_h := I^m_\Theta(\mu_h \bn) \in \bV_{h,\Theta}^{m}$. The following inequality holds:
\begin{equation*}
\Vert \bv_h \Vert_{A_h^L} \lesssim \Vert \mu_h \Vert_{M}.
\end{equation*}
\end{lemma}
\begin{proof}
Using the triangle inequality we get
\begin{equation} \label{eqinfsup4}
\Vert \bv_h \Vert_{A_h^L} \leq \Vert \mu_h \bn \Vert_{A_h^L} + \Vert  I^m_\Theta(\mu_h \bn) -  \mu_h \bn  \Vert_{A_h^L}.
\end{equation}
We estimate the two terms on the right hand side. The definition of the norm implies
\begin{equation} \label{eqinfsup5}
\Vert \mu_h \bn \Vert_{A_h^L}^2 = a_h(\mu_h \bn,\mu_h \bn) + s_h(\mu_h \bn,\mu_h \bn).
\end{equation}
The first term can be bounded by
\begin{equation}  \label{eqinfsup6} \begin{split}
& a_h(\mu_h \bn,\mu_h \bn)^{\frac{1}{2}} \lesssim \Vert \gradGh (\mu_h \bn) + \gradGh^T (\mu_h \bn) \Vert_{L^2(\Gamma_h)} + \Vert \mu_h \bn \Vert_{L^2(\Gamma_h)} \\
&\lesssim \Vert \gradGh (\mu_h \bn)  \Vert_{L^2(\Gamma_h)} + \Vert \mu_h \bn \Vert_{L^2(\Gamma_h)} \\ 
& \lesssim \Vert \bP_h (\bn  (\nabla \mu_h)^T  + \mu_h \nabla \bn) \bP_h  \Vert_{L^2(\Gamma_h)} + \Vert \mu_h \Vert_{L^2(\Gamma_h)} \\
&\lesssim \Vert (\bP_h - \bP) \bn  (\nabla \mu_h)^T \bP_h \Vert_{L^2(\Gamma_h)}  + \Vert \bP_h \mu_h \nabla \bn \bP_h  \Vert_{L^2(\Gamma_h)} + \Vert \mu_h \Vert_{M} \\ 
&\lesssim h^{k_g} \Vert \nabla \mu_h \Vert_{L^2(\Gamma_h)} + \Vert \mu_h \Vert_{M} \overset{ \eqref{HH4} }{\lesssim} (h^{k_g-1}+1)\|\mu_h\|_M\lesssim \|\mu_h\|_M.
\end{split}
\end{equation}
For the second term on the right hand side of equation \eqref{eqinfsup5} we get
\begin{align*}
s_h(\mu_h \bn,\mu_h \bn)^\frac{1}{2} &= \rho^\frac{1}{2} \Vert \nabla (\mu_h \bn) \bn_h \Vert_{L^2(\Omega_\Theta^\Gamma)} \\
&\leq \rho^\frac{1}{2} \left( \Vert \bn (\nabla \mu_h)\cdot \bn_h \Vert_{L^2(\Omega_\Theta^\Gamma)} + \Vert  \mu_h \nabla \bn \bn_h \Vert_{L^2(\Omega_\Theta^\Gamma)} \right) \\
&\overset{\nabla \bn \bn = 0}{\lesssim} \rho^\frac{1}{2} \left( \Vert \bn_h \cdot \nabla \mu_h \Vert_{L^2(\Omega_\Theta^\Gamma)} + \Vert  \mu_h \nabla \bn (\bn_h - \bn) \Vert_{L^2(\Omega_\Theta^\Gamma)} \right) \\
&\lesssim \Vert \mu_h \Vert_{M} + \rho^\frac{1}{2} h^{k_g} \Vert  \mu_h \Vert_{L^2(\Omega_\Theta^\Gamma)} 
\overset{\eqref{HH3}}{\lesssim} \Vert  \mu_h \Vert_{M}.
\end{align*}
Combining this with \eqref{eqinfsup6} we obtain
\begin{equation} \label{eqinfsup7}
\Vert \mu_h \bn \Vert_{A_h^L} \lesssim  \Vert  \mu_h \Vert_{M}.
\end{equation}
We now  consider the second term of the right hand side  \eqref{eqinfsup4}. Using $\vert \mu_h \vert_{H^{m+1}(\Theta_h(T))} = 0$ for all $T \in \mathcal{T}_h^\Gamma$ and componentwise the interpolation result \eqref{eqinterpolationerror} we get
\begin{align}
&\Vert  I^m_\Theta(\mu_h \bn) -  \mu_h \bn  \Vert_{A_h^L}^2 \nonumber  \\
&\hspace*{4.8mm}\lesssim \Vert \nabla( I^m_\Theta(\mu_h \bn) -  \mu_h \bn ) \Vert_{L^2(\Gamma_h)}^2 + \Vert  I^m_\Theta(\mu_h \bn) -  \mu_h \bn  \Vert_{L^2(\Gamma_h)}^2 \label{eqinfsup8}\\
&\hspace*{9.6mm}+ \rho \Vert \nabla( I^m_\Theta(\mu_h \bn) -  \mu_h \bn ) \bn_h \Vert_{L^2(\Omega_\Theta^\Gamma)}^2  \nonumber \\
&\hspace*{4.8mm}= \sum_{T \in \mathcal{T}_h^\Gamma} \Big( \Vert \nabla( I^m_\Theta(\mu_h \bn) -  \mu_h \bn ) \Vert_{L^2(\Gamma_T)}^2 + \Vert  I^m_\Theta(\mu_h \bn) -  \mu_h \bn  \Vert_{L^2(\Gamma_T)}^2 \nonumber \\
&\hspace*{9.6mm}+ \rho \Vert \nabla( I^m_\Theta(\mu_h \bn) -  \mu_h \bn ) \bn_h \Vert_{L^2(\Theta_h(T))}^2 \Big) \nonumber \\
&\hspace*{4.8mm}\overset{\eqref{eqtraceestimate}}{\lesssim} \sum_{T \in \mathcal{T}_h^\Gamma}
 \Big(h^{-1} \Vert  I^m_\Theta(\mu_h \bn) -  \mu_h \bn \Vert_{L^2(\Theta_h(T))}^2  \nonumber \\
&\hspace*{9.6mm} +
 (h^{-1} +h +\rho) \Vert  I^m_\Theta(\mu_h \bn) -  \mu_h \bn \Vert_{H^1(\Theta_h(T))}^2 + h \Vert  I^m_\Theta(\mu_h \bn) -  \mu_h \bn \Vert_{H^2(\Theta_h(T))}^2 \Big) \nonumber \\
&\hspace*{4.8mm}\lesssim \sum_{T \in \mathcal{T}_h^\Gamma} h^{2m-1}\Vert \mu_h \bn \Vert_{H^{m+1}(\Theta_h(T))}^2 
\lesssim \sum_{T \in \mathcal{T}_h^\Gamma} h^{2m-1}\Vert \mu_h  \Vert_{H^{m}(\Theta_h(T))}^2  \nonumber \\
&\hspace*{4.8mm}\overset{\text{inv. ineq.}}{\lesssim} \sum_{T \in \mathcal{T}_h^\Gamma} h^{-1}  \Vert \mu_h  \Vert_{L^2(\Theta_h(T))}^2 
\lesssim h^{-1}  \Vert \mu_h  \Vert_{L^2(\Omega_\Theta^\Gamma)}^2 
\overset{ \eqref{HH3}}{\lesssim}  \Vert \mu_h  \Vert_{M}^2. \nonumber 
\end{align} 
Combining this with  \eqref{eqinfsup4} and \eqref{eqinfsup7} we get the bound
$
\Vert \bv_h \Vert_{A_h^L} \lesssim \Vert \mu_h  \Vert_{M}$.
\end{proof}

Using these results one easily obtains the following discrete inf-sup property for $b_h(\cdot,\cdot)$.
\begin{lemma} \label{lemdiscreteinfsup}
Take $m \geq 1$. There exists a constant $c>0$, independent of $h$ and of how $\Gamma$  intersects the outer triangulation, such that, for $h$ sufficiently small
\begin{equation} \label{eqdiscreteinfsupcondition}
\sup_{\bv_h \in \bV_{h,\Theta}^{m}} \frac{b_h(\bv_h,\mu_h)}{\Vert \bv_h \Vert_{A_h^L}} \gtrsim \left( 1 - c \sqrt{\rho h} \right) \Vert \mu_h \Vert_{M} \qquad \text{for all } \mu_h \in V_{h,\Theta}^{m}.
\end{equation}
\end{lemma}
\begin{proof}
Take $\mu_h \in V_{h,\Theta}^{m}$ and define $\bv_h := I^m_\Theta(\mu_h \bn) \in \bV_{h,\Theta}^{m}$. Using Lemma \ref{lemmainfsup1} we get
\begin{align*}
 & \vert b_h(\bv_h,\mu_h) \vert \geq \vert b_h(\mu_h \bn,\mu_h) \vert - \vert b_h(I^m_\Theta(\mu_h \bn) - \mu_h \bn,\mu_h) \vert \\
&\gtrsim \Vert \mu_h \Vert_{M}^2 - \vert b_h(I^m_\Theta(\mu_h \bn) - \mu_h \bn,\mu_h) \vert \\
&\gtrsim \Vert \mu_h \Vert_{M}^2 - \Big( \Vert  I^m_\Theta(\mu_h \bn) -  \mu_h \bn  \Vert_{L^2(\Gamma_h)}^2 
+ \rho \Vert \nabla( I^m_\Theta(\mu_h \bn) -  \mu_h \bn ) \bn_h \Vert_{L^2(\Omega_\Theta^\Gamma)}^2 \Big)^\frac{1}{2} \Vert \mu_h \Vert_{M}.
\end{align*}
Following the estimates used in \eqref{eqinfsup8} one obtains
\begin{equation*} 
\Vert  I^m_\Theta(\mu_h \bn) -  \mu_h \bn  \Vert_{L^2(\Gamma_h)}^2 + \rho \Vert \nabla( I^m_\Theta(\mu_h \bn) -  \mu_h \bn ) \bn_h \Vert_{L^2(\Omega_\Theta^\Gamma)}^2 
\lesssim  \rho h \Vert \mu_h \Vert_{M}^2.
\end{equation*}
Combining these results with Lemma \ref{lemmainfsup2} we get
\begin{equation*}
\sup_{\bv_h \in \bV_{h,\Theta}^{m}} \frac{b_h(\bv_h,\mu_h)}{\Vert \bv_h \Vert_{A_h^L}} \gtrsim \left( 1 - c  \sqrt{\rho h} \right) \Vert \mu_h \Vert_{M} \qquad \text{for all } \mu_h \in V_{h,\Theta}^{m},
\end{equation*}
which completes the proof.
\end{proof}


\begin{corollary} \label{corollarydiscreteinfsup}
Take $m\geq 1$. Consider $\rho = c_\alpha h^{1-\alpha}$, $\alpha \in [0,2]$ and assume $h \leq h_0 \leq 1$. Take $c_\alpha$ such that $0 < c_\alpha < c^{-2} h_0^{\alpha-2}$ with $c$ as in \eqref{eqdiscreteinfsupcondition}. 
Then there exists a constant $d>0$, independent of $h$ and of how $\Gamma$ intersects the outer triangulation, such that:
\begin{equation*}
\sup_{\bv_h \in \bV_{h,\Theta}^{m}} \frac{b_h(\bv_h,\mu_h)}{\Vert \bv_h \Vert_{A_h^L}} \geq d  \Vert \mu_h \Vert_{M} \qquad \text{for all } \mu_h \in V_{h,\Theta}^{m}.
\end{equation*}
\end{corollary}

\begin{assumption} \label{assumptionrhoL}
We restrict to $\rho = c_\alpha h^{1-\alpha}$, $\alpha \in [0,2]$,  with $c_\alpha$ as in Corollary \ref{corollarydiscreteinfsup}. 
\end{assumption}

\begin{corollary}
Under Assumption \ref{assumptionrhoL} the discrete inf-sup property for $b_h(\cdot,\cdot)$ holds for the pair of spaces $(\bV_{h,\Theta}^{k},V_{h,\Theta}^{k_l})$ with $1\leq k_l \leq k$. The constant in the discrete inf-sup property estimate depends on $k_l$ but is independent of $h$ and of how $\Gamma$ intersects the outer triangulation.
\end{corollary}

From the fact that $A_h^L(\cdot,\cdot)$ defines a scalar product on $\bV_{h,\Theta}^{k}$, cf. Lemma~\ref{Ahscalarproduct}, and the discrete inf-sup property of $b_h(\cdot,\cdot)$ on $\bV_{h,\Theta}^{k} \times V_{h,\Theta}^{k_l}$ it follows that problem \eqref{discretelagrangeform} has a unique solution. Note that to show the discrete inf-sup property of $b_h(\cdot,\cdot)$ the stabilization $\tilde{s}_h(\cdot,\cdot)$ is essential.   

\subsection{Strang-Lemmas}

As usual, the discretization error analysis is based on a Strang Lemma which bounds the discretization error in terms of an approximation error and a consistency error.
We derive such Strang lemmas for the three discrete problems  \eqref{discretepenaltyform1}, \eqref{discretepenaltyform2} and \eqref{discretelagrangeform}. We first treat \eqref{discretepenaltyform1} and \eqref{discretepenaltyform2}.
\begin{theorem} \label{stranglemma}
For the unique solution $\bu=\bu_T^\ast\in \bV_T$ of problem \eqref{contform} and the unique solution $\bu_h \in \bV_{h,\Theta}^{k}$ of problem \eqref{discretepenaltyform1} respectively \eqref{discretepenaltyform2} the following discretization error bound holds for $i = 1,2$:
\begin{equation} \label{ineqstrang} \begin{split}
\Vert \bu^e - \bu_h \Vert_{A_h^{P_i}} &\leq 2 \min_{\bv_h \in \bV_{h,\Theta}^{k}} \Vert \bu^e - \bv_h \Vert_{A_h^{P_i}} \\
&\hspace{4.8mm}+ \sup_{\bw_h \in \bV_{h,\Theta}^{k}} \frac{\vert A_h^{P_i}(\bu^e, \bw_h) - (\bbf_h, \bw_h)_{L^2(\Gamma_h)} \vert}{\Vert \bw_h \Vert_{A_h^{P_i}}}.
\end{split}
\end{equation}
\end{theorem}
\begin{proof}
The proof uses standard arguments. For an arbitrary $\bv_h \in  \bV_{h,\Theta}^{k}$  we have
\begin{equation} \label{eqtriang}
\Vert \bu^e - \bu_h \Vert_{A_h^{P_i}} \leq \Vert \bu^e - \bv_h \Vert_{A_h^{P_i}} + \Vert \bv_h - \bu_h \Vert_{A_h^{P_i}}.
\end{equation}
Using the definition of the norm and setting $\bw_h = \bv_h - \bu_h \in \bV_{h,\Theta}^{k}$  results in 
\begin{equation*} \begin{split}
\Vert \bv_h - \bu_h \Vert_{A_h^{P_i}}^2 &= A_h^{P_i}(\bv_h - \bu_h, \bv_h - \bu_h) = A_h^{P_i}(\bv_h - \bu_h, \bw_h)  \\
&\leq \vert A_h^{P_i}(\bv_h - \bu^e, \bw_h) \vert + \vert A_h^{P_i}(\bu^e - \bu_h, \bw_h) \vert \\
&\leq \Vert \bu^e - \bv_h \Vert_{A_h^{P_i}} \Vert \bw_h \Vert_{A_h^{P_i}} + \vert A_h^{P_i}(\bu^e, \bw_h) -  (\bbf_h, \bw_h)_{L^2(\Gamma_h)} \vert .
\end{split}
\end{equation*}
Dividing by $\Vert \bw_h \Vert_{A_h^{P_i}} = \Vert  \bv_h - \bu_h \Vert_{A_h^{P_i}}$ together with inequality \eqref{eqtriang} completes the proof.
\end{proof}

For the analysis of Problem \eqref{discretelagrangeform} we define the bilinear form
\begin{equation*}
\mathcal{A}_h((\bu, \lambda),(\bv, \mu)) := A_h^L(\bu,\bv) + b_h(\bv, \lambda) + b_h(\bu, \mu), \qquad (\bu, \lambda),(\bv, \mu) \in \bV_{reg,h} \times V_{reg,h}.
\end{equation*}
From the well-posedness of the discrete problem \eqref{discretelagrangeform} it follows that $\mathcal{A}_h( \cdot,\cdot)$ fulfills a discrete inf-sup property, i.e.
\begin{equation} \label{eqinfsupAh}
\sup_{(\bv_h,\mu_h)  \in \bV_{h,\Theta}^{k} \times V_{h,\Theta}^{k_l}} \frac{\mathcal{A}_h((\bu_h,\lambda_h), (\bv_h,\mu_h)) }{\left(\Vert \bv_h \Vert_{A_h^L}^2 + \Vert \mu_h \Vert_{M}^2\right)^\frac{1}{2}} \gtrsim \left(\Vert \bu_h \Vert_{A_h^L}^2 + \Vert \lambda_h \Vert_{M}^2\right)^\frac{1}{2}
\end{equation}
for all $(\bu_h,\lambda_h) \in \bV_{h,\Theta}^{k} \times V_{h,\Theta}^{k_l}$. This will be used for a proof of the following Strang Lemma.

\begin{theorem} \label{stranglemmalagrange}
Let $(\bu, \lambda) = (\bu_T^\ast, \lambda) \in \bV_T \times L^2(\Gamma)$ be the unique solution of problem \eqref{contformlagrange} with $\bg:=\bbf$,  and $(\bu_h, \lambda_h) \in \bV_{h,\Theta}^{k} \times V_{h,\Theta}^{k_l}$ the unique solution of the discrete problem \eqref{discretelagrangeform}. The following discretization error bound holds:
\begin{equation} \begin{split}
&\Vert \bu^e - \bu_{h} \Vert_{A_h^L} + \Vert \lambda^e - \lambda_{h} \Vert_{M} \\  
&\hspace*{4.8mm}\lesssim \min_{(\bv_h,\mu_h)  \in \bV_{h,\Theta}^{k} \times V_{h,\Theta}^{k_l}} \left( \Vert \bu^e - \bv_h \Vert_{A_h^L} + \Vert \lambda^e - \mu_{h} \Vert_{M}\right) \\
&\hspace{9.6mm}+ \sup_{(\bw_h,\xi_h)  \in \bV_{h,\Theta}^{k} \times V_{h,\Theta}^{k_l}} \frac{\vert \mathcal{A}_h((\bu^e,\lambda^e), (\bw_h,\xi_h)) - (\bbf_h, \bw_h)_{L^2(\Gamma_h)} \vert}{\left(\Vert \bw_h \Vert_{A_h^L}^2 + \Vert \xi_h \Vert_{M}^2\right)^\frac{1}{2}}.
\end{split}
\end{equation}
\end{theorem}
\begin{proof} The discretization \eqref{discretelagrangeform} can be formulated in terms of the bilinear form $\mathcal{A}_h(\cdot,\cdot)$ on the product space
$\bV_{h,\Theta}^{k} \times V_{h,\Theta}^{k_l}$. 
Using  the discrete inf-sup property \eqref{eqinfsupAh} and the continuity of $\mathcal{A}_h(\cdot,\cdot)$ with respect to the product norm $(\|\cdot\|_{A_h^L}^2 +\|\cdot\|_M^2)^\frac12$ one can apply the same arguments as in the proof of Theorem~\ref{stranglemma}. 
\end{proof}

In the following two sections we analyze the approximation errors and the consistency errors, which appear in the Strang lemmas above.\\

\subsection{Approximation error bounds}

In the following lemma we show approximation error bounds in the norms that occur in the Strang lemmas above.
\begin{lemma} \label{lemmaapproximationerror}
For $\bu \in H^{k+1}(\Gamma)^3$ and $\lambda \in H^{k_l+1}(\Gamma)$ the following approximation error bounds hold:
\begin{equation} \label{Eq1}
 \min_{\bv_h \in \bV_{h,\Theta}^{k}} \Vert \bu^e - \bv_h \Vert_{A_h^{P_i}} \lesssim  (h^{k}  + \eta^{\frac{1}{2}} h^{k+1}) \Vert \bu \Vert_{H^{k+1}(\Gamma)}, \quad i =1,2
\end{equation}
\begin{equation} \label{Eq2}
 \begin{split}
& \min_{(\bv_h,\mu_h) \in \bV_{h,\Theta}^{k}\times V_{h,\Theta}^{k_l}} \left( \Vert \bu^e - \bv_h \Vert_{A_h^L} + \Vert \lambda^e - \mu_h \Vert_{M} \right) \\
&\hspace*{4.8mm}\lesssim  h^{k}  \Vert \bu \Vert_{H^{k+1}(\Gamma)} + (h^{k_l+1} + \rho^{\frac{1}{2}} h^{k_l+\frac{1}{2}}) \Vert \lambda \Vert_{H^{k_l+1}(\Gamma)}.
\end{split}
\end{equation}
\end{lemma}

\begin{proof}
We start with the $\Vert \cdot \Vert_{A_h^{P_1}}$-norm. Let $\bu \in H^{k+1}(\Gamma)^3$ and $\bw_h := I^{k}_{\Theta} (\bu^e)$ the component-wise parametric interpolation. We then have 
\begin{equation*} \begin{split}
 & \min_{\bv_h \in \bV_{h,\Theta}^{k}} \Vert \bu^e - \bv_h \Vert_{A_h^{P_1}}^2 \leq  \Vert \bu^e - \bw_h \Vert_{A_h^{P_1}}^2 \\
 &= a_h(\bu^e - \bw_h,\bu^e - \bw_h) + s_h(\bu^e - \bw_h,\bu^e - \bw_h) 
  + k_h(\bu^e - \bw_h,\bu^e - \bw_h).
\end{split}
\end{equation*}
For the first term we get using component-wise Lemma \ref{lemmascalarapproximationerror}
\begin{equation} \label{eqapproximationerror1}\begin{split}
 & a_h(\bu^e - \bw_h,\bu^e - \bw_h) \lesssim \left\Vert E_h(\bu^e - \bw_h) \right\Vert_{L^2(\Gamma_h)}^2 +  \Vert \bu^e - \bw_h \Vert_{L^2(\Gamma_h)}^2 \\
&\lesssim \Vert \nabla (\bu^e - \bw_h) \Vert_{L^2(\Gamma_h)}^2 +  \Vert \bu^e - \bw_h \Vert_{L^2(\Gamma_h)}^2 \lesssim h^{2k}\Vert \bu \Vert_{H^{k+1}(\Gamma)}^2.
\end{split}
\end{equation}
The second term leads to 
\begin{equation} \label{eqapproximationerror2} \begin{split}
 & s_h(\bu^e - \bw_h,\bu^e - \bw_h) = \rho \Vert \nabla (\bu^e - \bw_h) \bn_h \Vert_{L^2(\Omega^\Gamma_{\Theta})}^2 \\
&\leq \rho \Vert \bu^e - \bw_h \Vert_{H^1(\Omega^\Gamma_{\Theta})}^2 
\overset{\eqref{eqinterpolationerror}}{\lesssim} \rho h^{2k}\Vert \bu^e \Vert_{H^{k+1}(\Omega^\Gamma_{\Theta})}^2 
\overset{\eqref{lemmasobolevnormsneighborhood}}{\lesssim} \rho h^{2k+1}\Vert \bu \Vert_{H^{k+1}(\Gamma)}^2.
\end{split}
\end{equation}
For the third term we obtain
\begin{equation*} \begin{split}
k_h(\bu^e - \bw_h,\bu^e - \bw_h) & = \eta \Vert (\bu^e - \bw_h) \cdot \tilde{\bn}_h \Vert_{L^2(\Gamma_h)}^2 \\
& \lesssim \eta \Vert \bu^e - \bw_h \Vert_{L^2(\Gamma_h)}^2 
\overset{\text{Lemma } \ref{lemmascalarapproximationerror}}{\lesssim} \eta h^{2(k+1)}\Vert \bu \Vert_{H^{k+1}(\Gamma)}^2.
\end{split}
\end{equation*} 
Combining this with \eqref{eqapproximationerror1}, \eqref{eqapproximationerror2} and $\rho \lesssim h^{-1}$ proves the bound for the $\Vert \cdot \Vert_{A_h^{P_1}}$-norm. Since 
\begin{equation*} \begin{split}
a_{T,h}(\bu^e - \bw_h,\bu^e - \bw_h) &\lesssim \left\Vert E_{T,h}(\bu^e - \bw_h) \right\Vert_{L^2(\Gamma_h)}^2 +  \Vert \bP_h( \bu^e - \bw_h) \Vert_{L^2(\Gamma_h)}^2 \\
&\lesssim \left\Vert E_h(\bu^e - \bw_h) \right\Vert_{L^2(\Gamma_h)}^2 +  \Vert \bu^e - \bw_h \Vert_{L^2(\Gamma_h)}^2
\end{split}
\end{equation*}
we also immediately get the bound for the $\Vert \cdot \Vert_{A_h^{P_2}}$-norm. Hence, the result in \eqref{Eq1} holds. Now derive the result \eqref{Eq2}. Since 
\begin{equation*} \begin{split}
\Vert \bu^e - \bw_h \Vert_{A_h^L}^2 = a_h(\bu^e - \bw_h,\bu^e - \bw_h) + s_h(\bu^e - \bw_h,\bu^e - \bw_h),
\end{split}
\end{equation*}
we use the estimates in \eqref{eqapproximationerror1} and \eqref{eqapproximationerror2}.
To show the approximation error bound in the $\Vert \cdot \Vert_{M}$-norm we take $\lambda \in H^{k_l+1}(\Gamma)$ and define $\xi_h := I^{k_l}_{\Theta} (\lambda^e)$. Then we have
\begin{equation*} \begin{split}
 \min_{\mu_h \in \bV_{h,\Theta}^{k_l}} \Vert \lambda^e - \mu_h \Vert_{M}  & \leq  \Vert \lambda^e - \xi_h \Vert_{M} \\
 &\lesssim \Vert \lambda^e - \xi_h \Vert_{L^2(\Gamma_h)} + \rho^\frac{1}{2}\Vert \bn_h \cdot \nabla (\lambda^e - \xi_h) \Vert_{L^2(\Omega_\Theta^\Gamma)} \\
 &\overset{\text{Lemma } \ref{lemmascalarapproximationerror}}{\lesssim} h^{k_l+1} \Vert \lambda \Vert_{H^{k_l+1}(\Gamma)} + \rho^\frac{1}{2}\Vert \lambda^e - \xi_h \Vert_{H^1(\Omega_\Theta^\Gamma)}  \\ & 
\overset{\eqref{eqinterpolationerror}, \eqref{lemmasobolevnormsneighborhood}}{\lesssim} (h^{k_l+1} + \rho^\frac{1}{2} h^{k_l})\Vert \lambda \Vert_{H^{k_l+1}(\Gamma)},
\end{split}
\end{equation*}
which completes the proof.
\end{proof}

Note that in \eqref{Eq1}, \eqref{Eq2} we obtain optimal order approximation errors, provided $\rho \lesssim h^{-1}$ and $\eta \lesssim h^{-2}$. 
\subsection{Consistency error analysis}

In this section we present a consistency error analysis. The analysis is rather long and technical. The structure is a follows. In section~\ref{preliminaries} we collect a few basic results for vector functions $\bu \in H^1(\Gamma)^3$ and corresponding extensions $\bu^e \in H^1(\Gamma_h)^3$. These results are rather straightforward and very similar to known results for scalar surface functions. In section \ref{sectgeometry} we derive bounds for basic components of the consistency error that are directly related to the geometry approximation $\Gamma_h\approx \Gamma$. We derive, for example, a bound for $|a_h(\bv,\bw)-a(\bv^l,\bw^l)|$. A key result is derived in section~\ref{SectFEkorn}, namely a discrete Korn-type inequality. Using these preparations, the consistency bounds for the three methods are derived in the sections~\ref{sectcons1} and \ref{sectcons2}.    

\subsubsection{Preliminaries} \label{preliminaries}
We start with results concerning the transformation of the integrals between $\Gamma$ and $\Gamma_h$. Using $\nabla p = \bP - d \bH$ we get for $u \in H^1(\Gamma)$ and  $x \in \Gamma_h$ 
\begin{equation} \label{transfo1} \begin{split}
\gradGh u^e(x) &= \gradGh (u \circ p)(x) 
= \bP_h(x) \nabla p(x) \nabla u(p(x)) \\
&= \bP_h(x) (\bP(x) - d(x)\bH(x))\nabla u(p(x)) = \bB^T(x) \gradG u(p(x)),
\end{split}
\end{equation}
 with $\bB=\bB(x) := \bP(\bI - d\bH)\bP_h$ ($x \in \Gamma_h$). 

From \cite{hansbo2016analysis} we have the following Lemma:
\begin{lemma} \label{lemmaB}
For $x \in \Gamma_h$ and $\bB=\bB(x)$ as above, the map $\bB_{|{\rm range} (\bP_h(x))}$ is invertible for $h$ small enough, i.e. there is $\bB^{-1} \colon {\rm range} (\bP(x)) \to {\rm range} (\bP_h(x))$ such that
\begin{equation*}
\bB\bB^{-1} = \bP, \qquad \bB^{-1} \bB = \bP_h
\end{equation*} 
and we have for $u \in H^1(\Gamma)$, $x \in \Gamma_h$,
\begin{equation*}
\gradG u(p(x)) = \bP(x)\bB^{-T}(x)\gradGh u^e(x).
\end{equation*}
Furthermore, the following estimates hold:
\begin{align*}
\Vert \bB \Vert_{L^\infty(\Gamma_h)} &\lesssim 1, \qquad &\Vert \bP_h\bB^{-1} \bP \Vert_{L^\infty(\Gamma_h)} &\lesssim 1 \\
\Vert \bP \bP_h - \bB \Vert_{L^\infty(\Gamma_h)} &\lesssim h^{k_g+1}, \qquad &\Vert \bP_h \bP - \bP_h\bB^{-1}\bP \Vert_{L^\infty(\Gamma_h)} &\lesssim h^{k_g+1}. 
\end{align*}
For the surface measures on $\Gamma$ and $\Gamma_h$ we have the identity 
\begin{equation*}
d\Gamma = \vert \bB \vert d\Gamma_h
\end{equation*}
where $\vert \bB \vert = \vert det(\bB) \vert$ and we have the estimates
\begin{equation*}
\Vert 1- \vert \bB \vert \Vert_{L^{\infty}(\Gamma_h)} \lesssim h^{k_g+1}, \quad \Vert \vert \bB \vert \Vert_{L^{\infty}(\Gamma_h)} \lesssim 1 , \quad \Vert \vert \bB \vert^{-1} \Vert_{L^{\infty}(\Gamma_h)} \lesssim 1.
\end{equation*}
\end{lemma}
Applying Lemma \ref{lemmaB} yields, for $u \in H^1(\Gamma)$, 
\begin{equation*}
\gradG u^l(p(x)) = \bP(x)\bB^{-T}(x)\gradGh u(x), \quad x \in \Gamma_h.
\end{equation*}
Similar useful transformation results for vector-valued functions are given in the following corollary.
\begin{corollary} \label{corollarygradients}
For $\bu \in H^1(\Gamma)^3$ and $\bv \in H^1(\Gamma_h)^3$ we have
\begin{equation*} \begin{split}
\left(\nabla \bu^e \bP\right)^e &= \nabla \bu^e \bP = \nabla \bu^e \bP_h \bB^{-1}\bP \quad \text{on } \Gamma_h, \\
\left(\nabla \bv^l \bP \right)^e &= \nabla \bv^l \bP = \nabla \bv^l \bP_h \bB^{-1} \bP \quad \text{on } \Gamma_h.
\end{split}
\end{equation*}
\end{corollary}
\begin{proof}
For $\bu \in H^1(\Gamma)^3$ we have with \eqref{transfo1} and Lemma \ref{lemmaB}
\begin{equation*} \begin{split}
\be_i^T \nabla \bu^e \bP_h &= (\nabla u_i^e)^T \bP_h = (\gradGh u_i^e)^T = (\bB^T \gradG u_i \circ p)^T \\
&= (\bB^T \bP \nabla u_i^e)^T = \be_i^T \nabla \bu^e \bP \bB \quad \text{on } \Gamma_h
\end{split}
\end{equation*}
for $i =1,2,3$. Multiplying by $\bB^{-1} \bP$ from the right results in the equation above. For $\bv \in H^1(\Gamma_h)^3$ we use similar arguments:
\begin{equation*} \begin{split}
\be_i^T \nabla \bv^l \bP_h &= (\nabla v_i^l)^T \bP_h = (\gradGh v_i)^T = (\bB^T \gradG v_i^l \circ p)^T \\
&= (\bB^T \bP \nabla v_i^l)^T = \be_i^T \nabla \bv_i^l \bP \bB \quad \text{on } \Gamma_h.
\end{split}
\end{equation*}
for $i =1,2,3$. Multiplying by $\bB^{-1} \bP$ from the right completes the proof.
\end{proof}

For scalar-valued functions $w \in H^1(\Gamma_h)$ the following equivalences are well known (see \cite{Dziuk88}):
\begin{equation*}
\Vert w \Vert_{L^2(\Gamma_h)} \sim \Vert w^l \Vert_{L^2(\Gamma)}, \qquad \Vert \nabla_{\Gamma_h} w \Vert_{L^2(\Gamma_h)} \sim \Vert \nabla_\Gamma w^l \Vert_{L^2(\Gamma)}.
\end{equation*}
We need similar equivalences for vector-valued functions. These are given in the following lemma.
\begin{lemma} \label{lemmanormequivalences}
For $\bv \in H^1(\Gamma_h)^3$ we have
\begin{align} 
\Vert \bv \Vert_{L^2(\Gamma_h)} &\sim \Vert \bv^l \Vert_{L^2(\Gamma)}, \label{eqnormequivalenceL2} \\
\Vert \nabla \bv^l \bP_h \Vert_{L^2(\Gamma_h)} &\sim \Vert \nabla \bv^l \bP \Vert_{L^2(\Gamma)}. \label{eqnormequivalenceH1}
\end{align}
\end{lemma}
\begin{proof}
Let $\bv \in H^1(\Gamma_h)^3$. We start with the first equivalence. Using the definition of the lifting, i.e. $\bv^l(p(x)) = \bv(x)$ for $x \in \Gamma_h$, and the integral transformation rule, with Lemma \ref{lemmaB} we  obtain \eqref{eqnormequivalenceL2}. For the second norm equivalence we use Corollary \ref{corollarygradients}:
\begin{equation*} \begin{split}
\Vert \nabla \bv^l \bP_h \Vert_{L^2(\Gamma_h)}^2 &= \int_{\Gamma_h} \nabla \bv^l \bP_h : \nabla \bv^l \bP_h \, ds_h 
= \int_{\Gamma_h} \left(\nabla \bv^l \bP \right)^e \bB : \left(\nabla \bv^l \bP \right)^e \bB \, ds_h \\
&= \int_{\Gamma} \nabla \bv^l \bP (\bB \circ p^{-1}) : \nabla \bv^l \bP (\bB \circ p^{-1}) \vert \bB \vert^{-1} \circ p^{-1} \, ds \\
&\lesssim \Vert \bB \circ p^{-1} \Vert_{L^\infty(\Gamma)}^2 \Vert \vert \bB \vert^{-1} \circ p^{-1} \Vert_{L^\infty(\Gamma)} \Vert \nabla \bv^l \bP  \Vert_{L^2(\Gamma)}^2 
\lesssim \Vert \nabla \bv^l \bP \Vert_{L^2(\Gamma)}^2,
\end{split}
\end{equation*}
where $p^{-1}$ is the inverse of $p|_{\Gamma_h}$. The other direction is obtained with similar arguments:
\begin{equation*} \begin{split}
\Vert \nabla \bv^l \bP \Vert_{L^2(\Gamma)}^2 &= \int_{\Gamma} \nabla \bv^l \bP : \nabla \bv^l \bP \, ds 
=\int_{\Gamma_h} \left(\nabla \bv^l \bP\right)^e : \left(\nabla \bv^l \bP\right)^e \vert \bB \vert \, ds_h \\
&=\int_{\Gamma_h} \nabla \bv^l \bP_h \bB^{-1} \bP : \nabla \bv^l \bP_h \bB^{-1} \bP \vert \bB \vert \, ds_h \\
&\lesssim \Vert \bP_h\bB^{-1} \bP \Vert_{L^\infty(\Gamma_h)}^2 \Vert \vert \bB \vert \Vert_{L^\infty(\Gamma_h)} \Vert \nabla \bv^l \bP_h \Vert_{L^2(\Gamma_h)}^2 
\lesssim \Vert \nabla \bv^l \bP_h \Vert_{L^2(\Gamma_h)}^2. 
\end{split}
\end{equation*}
\end{proof}
\subsubsection{Geometry errors} \label{sectgeometry}
In this section we analyze certain  parts of the consistency error, which are similar in the three discretizations. For this we introduce further notation. 
We define, for $\bv, \bw \in \bV_{reg,h}$:
\begin{align*}
 G_a(\bv,\bw) &:= a_h(\bv,\bw) - a(\bv^l, \bw^l),~~~G_{a_T}(\bv,\bw) := a_{T,h}(\bv,\bw) - a_T(\bv^l, \bw^l),\\
G_f(\bw)   & := (\bbf,\bw^l)_{L^2(\Gamma)} - (\bbf_h,\bw)_{L^2(\Gamma_h)}.
\end{align*}
Let $\bu=\bu_T$ be the solution of \eqref{contform} and $\bw_h \in \bV_{h,\Theta}^{k}$.
The consistency term corresponding to \eqref{discretepenaltyform1} can be written as
\begin{align} 
&A_h^{P_1}(\bu^e, \bw_h) - (\bbf_h, \bw_h)_{L^2(\Gamma_h)} \nonumber \\
&\hspace*{4.8mm}= a_h(\bu^e,\bw_h) + s_h(\bu^e,\bw_h) + k_h(\bu^e,\bw_h) - (\bbf_h,\bw_h)_{L^2(\Gamma_h)} \nonumber \\
&\hspace*{9.6mm}\underbrace{- a(\bu, \bP \bw_h^l) + (\bbf,\bw_h^l)_{L^2(\Gamma)}}_{=0}\nonumber \\
&\hspace*{4.8mm}= a_h(\bu^e,\bw_h) - a(\bu, \bw_h^l) + s_h(\bu^e,\bw_h) + k_h(\bu^e,\bw_h)  \label{consterm1}  \\
&\hspace*{9.6mm} + ( E(\bu) , E((\bw_h^l \cdot \bn)\bn))_{L^2(\Gamma)} + (\bbf,\bw_h^l)_{L^2(\Gamma)} - (\bbf_h,\bw_h)_{L^2(\Gamma_h)} \nonumber \\
&\hspace*{4.8mm}= G_{a}(\bu^e, \bw_h) + ( E(\bu) , E((\bw_h^l \cdot \bn)\bn))_{L^2(\Gamma)} + s_h(\bu^e,\bw_h) + k_h(\bu^e,\bw_h) + G_f(\bw_h).\nonumber 
\end{align}
Similarly, for \eqref{discretepenaltyform2} we get
\begin{equation} \label{consterm2} \begin{split}
&A_h^{P_2}(\bu^e, \bw_h) - (\bbf_h, \bw_h)_{L^2(\Gamma_h)} \\
&\hspace*{4.8mm}= G_{a_T}(\bu^e, \bw_h)  + s_h(\bu^e,\bw_h) + k_h(\bu^e,\bw_h) + G_f(\bw_h).
\end{split}
\end{equation}
Let $(\bu,\lambda)$ be the solution of problem \eqref{contformlagrange}. With $(\bw_h,\mu_h)  \in \bV_{h,\Theta}^{k} \times V_{h,\Theta}^{k_l}$ we get,
\begin{align}
&\mathcal{A}_h((\bu^e,\lambda^e), (\bw_h,\mu_h)) - (\bbf_h, \bw_h)_{L^2(\Gamma_h)} \nonumber \\
&\hspace*{4.8mm}= a_h(\bu^e,\bw_h) + b_h(\bw_h,\lambda^e) + b_h(\bu^e, \mu_h) +  s_h(\bu^e,\bw_h) - (\bbf_h, \bw_h)_{L^2(\Gamma_h)}\label{consterm3}  \\
&\hspace*{9.6mm}+ \underbrace{(\bbf, \bw_h^l)_{L^2(\Gamma)} - a(\bu,\bw_h^l) - (\bw_h^l \cdot \bn, \lambda)_{L^2(\Gamma)}}_{=0} \nonumber \\
&\hspace*{4.8mm}= G_a(\bu^e, \bw_h) + b_h(\bw_h,\lambda^e) + b_h(\bu^e, \mu_h) +  s_h(\bu^e,\bw_h) - (\bw_h^l \cdot \bn, \lambda)_{L^2(\Gamma)} + G_f(\bw_h).
\nonumber
\end{align}
For the derivation of bounds for the geometry errors $G_a(\cdot,\cdot)$, $G_{a_T}(\cdot,\cdot)$ and $G_f(\cdot)$ we use the following lemma. We use the notation
$E_T(\bw):=E(\bP\bw)=E(\bw)- w_N \bH$, for $\bw \in H^1(\Gamma)^3$.
\begin{lemma} \label{lemmaDEdifference}
For $\bv \in H^1(\Gamma_h)^3$ the following bounds hold:
\begin{equation} \label{eqdiffoperators} \begin{split}
\Vert (\gradG \bv^l)^e - \gradGh \bv \Vert_{L^2(\Gamma_h)} &\lesssim h^{k_g}  \Vert  \bv \Vert_{H^1(\Gamma_h)}, \\
\Vert (E (\bv^l))^e - E_h( \bv) \Vert_{L^2(\Gamma_h)} &\lesssim h^{k_g} \Vert  \bv \Vert_{H^1(\Gamma_h)}, \\
\Vert (E_T (\bv^l))^e - E_{T,h}( \bv) \Vert_{L^2(\Gamma_h)} &\lesssim h^{k_g} \left( \Vert  \bv \Vert_{H^1(\Gamma_h)} + h^{-1} \Vert \bv \cdot \bn_h \Vert_{L^2(\Gamma_h)} \right).
\end{split}
\end{equation}
\end{lemma}

\begin{proof}
We start with the first inequality. Using Corollary \ref{corollarygradients} we can write
\begin{equation*} \begin{split}
(\gradG \bv^l)^e - \gradGh \bv &= (\bP \nabla \bv^l \bP)^e - \bP_h \nabla \bv^l \bP_h \\
&= \bP \nabla \bv^l \bP_h \bB^{-1} \bP - \bP_h \nabla \bv^l \bP_h \\
&= (\bP - \bP_h) \nabla \bv^l \bP_h \bB^{-1} \bP + \bP_h \nabla \bv^l \bP_h (\bP_h\bB^{-1} \bP - \bP_h\bP) \\
&\hspace*{4.8mm}+ \bP_h \nabla \bv^l \bP_h (\bP - \bP_h). 
\end{split}
\end{equation*} 
Hence, with Lemma \ref{lemmaB} we get
\begin{equation*} \begin{split}
\Vert (\gradG \bv^l)^e - \gradGh \bv \Vert_{L^2(\Gamma_h)} &\lesssim \Vert \bP - \bP_h \Vert_{L^\infty(\Gamma_h)} \Vert \nabla \bv^l \bP_h \Vert_{L^2(\Gamma_h)}  \Vert \bP_h\bB^{-1} \bP \Vert_{L^\infty(\Gamma_h)} \\
&\hspace*{4.8mm}+ \Vert  \bP_h \Vert_{L^\infty(\Gamma_h)}  \Vert \nabla \bv^l \bP_h \Vert_{L^2(\Gamma_h)} \Vert\bP_h \bB^{-1} \bP - \bP_h\bP \Vert_{L^\infty(\Gamma_h)} \\
&\hspace*{4.8mm}+ \Vert \bP_h \Vert_{L^\infty(\Gamma_h)} \Vert \nabla \bv^l \bP_h \Vert_{L^2(\Gamma_h)} \Vert \bP - \bP_h \Vert_{L^\infty(\Gamma_h)} \\
&\lesssim h^{k_g} \Vert \nabla \bv^l \bP_h \Vert_{L^2(\Gamma_h)},
\end{split}
\end{equation*}
which shows the first inequality in \eqref{eqdiffoperators}. Combining this with
\begin{equation*} \begin{split}
(E (\bv^l))^e - E_h( \bv) &= \frac{1}{2} \left( (\gradG \bv^l)^e + (\gradG^T \bv^l)^e  \right) - \frac{1}{2} \left( \gradGh \bv + \gradGh^T \bv  \right) \\
&= \frac{1}{2} \left( (\gradG \bv^l)^e - \gradGh \bv  \right) + \frac{1}{2} \left( (\gradG^T \bv^l)^e - \gradGh^T \bv  \right)
\end{split}
\end{equation*}
we obtain the second inequality in \eqref{eqdiffoperators}.   For the last inequality in \eqref{eqdiffoperators} we note
\begin{equation*}
(E_T (\bv^l))^e - E_{T,h}( \bv) = (E (\bv^l))^e - E_h( \bv) - ((\bv \cdot \bn)\bH)^e + (\bv \cdot \bn_h)\bH_h.
\end{equation*}
  Applying Lemma \ref{lemmanormals} and inequality \eqref{eqapproxH} we obtain
\begin{equation*} \begin{split}
&\Vert ((\bv \cdot \bn)\bH)^e - (\bv \cdot \bn_h)\bH_h  \Vert_{L^2(\Gamma_h)} \\
&\hspace*{4.8mm}\lesssim \Vert (\bv \cdot(\bn - \bn_h))\bH  \Vert_{L^2(\Gamma_h)} + \Vert (\bv \cdot \bn_h)(\bH -\bH_h)  \Vert_{L^2(\Gamma_h)} \\
&\hspace*{4.8mm}\lesssim h^{k_g} \Vert \bv \Vert_{L^2(\Gamma_h)} + h^{k_g-1} \Vert \bv \cdot \bn_h \Vert_{L^2(\Gamma_h)} \lesssim h^{k_g} \left(\Vert \bv \Vert_{L^2(\Gamma_h)} + h^{-1} \Vert \bv \cdot \bn_h \Vert_{L^2(\Gamma_h)}\right).
\end{split}
\end{equation*}
Combining this with the second inequality in \eqref{eqdiffoperators} we obtain the third one.
\end{proof}
\ \\
\begin{lemma} \label{lemmageometryapproximationerrors}
Let $\bbf_h$ be an approximation of $\bbf$ such that $\Vert \vert \bB\vert \bbf^e - \bbf_h \Vert_{L^2(\Gamma_h)} \lesssim h^{k_g+1} \Vert \bbf \Vert_{L^2(\Gamma)}$ holds. For $\bv, \bw \in H^1(\Gamma_h)^3$ we then have 
\begin{equation*} \label{eqgeometryapproximationerrors} \begin{split}
\vert G_{a_T}(\bv,\bw) \vert &\lesssim h^{k_g} \left(\Vert \bv \Vert_{H^1(\Gamma_h)} + h^{-1} \Vert \bv \cdot \bn_h \Vert_{L^2(\Gamma_h)}\right) \left(\Vert \bw \Vert_{H^1(\Gamma_h)} + h^{-1} \Vert \bw \cdot \bn_h \Vert_{L^2(\Gamma_h)}\right), \\
\vert G_a(\bv,\bw) \vert &\lesssim h^{k_g} \Vert \bv \Vert_{H^1(\Gamma_h)} \Vert \bw \Vert_{H^1(\Gamma_h)}, \\
\vert G_f(\bw) \vert &\lesssim h^{k_g+1} \Vert \bbf \Vert_{L^2(\Gamma)} \Vert \bw \Vert_{L^2(\Gamma_h)}.
\end{split}
\end{equation*}
\end{lemma}
\begin{proof}
We start with the estimate for the geometric error $G_{a_T}(\cdot,\cdot)$. Using the definitions we can write
\begin{equation} \label{eqproofGa} \begin{split}
G_{a_T}(\bv,\bw) &= a_{T,h}(\bv,\bw) - a_T(\bv^l, \bw^l) \\ 
&= (E_{T,h}(\bv), E_{T,h}(\bw))_{L^2(\Gamma_h)} - (\bP_h\bv,\bP_h\bw)_{L^2(\Gamma_h)} \\
&\hspace*{4.8mm}- (E_T(\bv^l), E_T(\bw^l))_{L^2(\Gamma)} + (\bP \bv^l, \bP\bw^l)_{L^2(\Gamma)}.
\end{split}
\end{equation}
Combining the second and fourth term we get 
\begin{align*}
 (\bP_h\bv,\bP_h\bw)_{L^2(\Gamma_h)} - (\bP\bv^l,\bP\bw^l)_{L^2(\Gamma)}  
  =  (\bP_h\bv,\bP_h\bw)_{L^2(\Gamma_h)} - (\vert \bB \vert \bP\bv,\bP\bw)_{L^2(\Gamma_h)}.
\end{align*}
Using an obvious splitting, $\|\bP- \bP_h\|_{L^\infty(\Gamma_h)}\lesssim h^{k_g}$ and $\|1- |\bB|\|_{L^\infty(\Gamma_h)}\lesssim h^{k_g+1}$ we obtain a bound $\lesssim h^{k_g}\Vert \bv \Vert_{L^2(\Gamma_h)}  \Vert \bw \Vert_{L^2(\Gamma_h)}$. 
For the first and third term of the right hand side of equation \eqref{eqproofGa} we have
\begin{equation*} \begin{split}
& (E_{T,h}(\bv), E_{T,h}(\bw))_{L^2(\Gamma_h)} -  (E_T(\bv^l), E_T(\bw^l))_{L^2(\Gamma)} \\
&\hspace*{4.8mm}=  (E_{T,h}(\bv), E_{T,h}(\bw))_{L^2(\Gamma_h)} - (\vert \bB \vert(E_T(\bv^l))^e, (E_T(\bw^l))^e)_{L^2(\Gamma_h)}.
\end{split}
\end{equation*}
Using a similar splitting, the third inequality in Lemma~\ref{lemmaDEdifference}, $ \Vert (E_T(\bv^l))^e \Vert_{L^2(\Gamma_h)} \lesssim \Vert  \bv \Vert_{H^1(\Gamma_h)}$, $\Vert (E_T(\bw^l))^e \Vert_{L^2(\Gamma_h)} \lesssim \Vert  \bw \Vert_{H^1(\Gamma_h)}$ and combining this with the result above we obtain the bound for $G_{a_T}(\cdot,\cdot)$.
For
\begin{equation*} \begin{split}
G_a(\bv,\bw) &= a_h(\bv,\bw) - a(\bv^l, \bw^l)  \\ 
&=  (E_h(\bv), E_h(\bw))_{L^2(\Gamma_h)} - (\bv,\bw)_{L^2(\Gamma_h)}\\
&\hspace*{4.8mm}- (E(\bv^l), E(\bw^l))_{L^2(\Gamma)} + (\bv^l, \bw^l)_{L^2(\Gamma)}
\end{split}
\end{equation*}
very similar arguments can be applied.
Finally, the  bound for $G_f(\cdot)$ follows from
\begin{equation*} \begin{split}
|G_f(\bw)| & = |(\bbf,\bw^l)_{L^2(\Gamma)} - (\bbf_h,\bw)_{L^2(\Gamma_h)}| 
=\left|  (\vert \bB \vert \bbf^e ,\bw)_{L^2(\Gamma_h)} - (\bbf_h,\bw)_{L^2(\Gamma_h)} \right| \\
&\lesssim \Vert \vert \bB\vert \bbf^e - \bbf_h \Vert_{L^2(\Gamma_h)} \Vert \bw \Vert_{L^2(\Gamma_h)} 
\lesssim h^{k_g+1} \Vert \bbf \Vert_{L^2(\Gamma)} \Vert \bw \Vert_{L^2(\Gamma_h)}.
\end{split}
\end{equation*}
\end{proof}
\begin{remark} \rm If in Lemma~\ref{lemmageometryapproximationerrors}, for $\bv, \bw$ we take $\bu^e, \bw^e$, with smooth function $\bu,\bw \in H^2(\Gamma)$, it may be possible to improve the bounds. This can be relevant in the derivation of $L^2$-norm optimal error bounds (which we do not consider in this paper).
\end{remark}

\subsection{Discrete Korn's type inequality} \label{SectFEkorn}
In  geometry error bounds derived in  Lemma~\ref{lemmageometryapproximationerrors} we obtain natural terms of the form $\|\bw_h\|_{H^1(\Gamma_h)}$, with $\bw_h \in \bV_{h,\Theta}^{k}$. These have to be controlled in terms of the discrete energy norms, cf. Strang-Lemmas.  
A key tool for quantifying this control is a discrete Korn's type inequality that is derived in this section. This result can be understood as an analogon of so-called discrete $H^1$ type bounds derived for higher order surface finite element spaces in \cite{hansbo2016analysis}.
\begin{lemma} \label{lemmadiscretekorninequality}
For $h$ sufficiently small the following holds:
\begin{equation} \label{discKorn1}\begin{split} 
\Vert \bv_h \Vert_{H^1(\Gamma_h)} &\lesssim \Vert E_{T,h}(\bv_h) \Vert_{L^2(\Gamma_h)} + \Vert \bP_h \bv_h \Vert_{L^2(\Gamma_h)} + h^{-1} \Vert \bv_h \cdot {\bn}_h \Vert_{L^2(\Gamma_h)} \\
&\hspace*{4.8mm}+ h^{-\frac{1}{2}} \Vert \nabla \bv_h \bn_h \Vert_{L^2(\Omega_\Theta^\Gamma)}, \quad \text{for all}~~\bv_h \in \bV_{h,\Theta}^{k}.
\end{split}
\end{equation}
\end{lemma}
\begin{proof}
From Lemma \ref{lemmanormequivalences} it follows
\begin{equation} \label{eqdiscreteKornsplit}
\Vert \bv_h \Vert_{H^1(\Gamma_h)} \lesssim \Vert \bv_h^l \Vert_{H^1(\Gamma)} \lesssim \Vert \bP \bv_h^l \Vert_{H^1(\Gamma)} + \Vert  \bv_h^l \cdot \bn \Vert_{H^1(\Gamma)}.
\end{equation}
The term with the tangential part, $\Vert \bP \bv_h^l \Vert_{H^1(\Gamma)}$, can be bounded using the surface Korn inequality (Lemma \ref{lemmakorn})  and Lemma \ref{lemmaDEdifference}:
\[ \begin{split}
& \Vert \bP \bv_h^l \Vert_{H^1(\Gamma)} 
\lesssim \Vert E(\bP \bv_h^l) \Vert_{L^2(\Gamma)} + \Vert \bP \bv_h^l \Vert_{L^2(\Gamma)} 
= \Vert E_T(\bv_h^l) \Vert_{L^2(\Gamma)} + \Vert \bP \bv_h^l \Vert_{L^2(\Gamma)} \\
&\hspace*{4.8mm}\lesssim \Vert \left(E_T(\bv_h^l) \right)^e \Vert_{L^2(\Gamma_h)} + \Vert \bP \bv_h \Vert_{L^2(\Gamma_h)} \\
&\hspace*{4.8mm}\lesssim \Vert E_{T,h}(\bv_h) \Vert_{L^2(\Gamma_h)} + \Vert \left(E_T(\bv_h^l) \right)^e - E_{T,h}(\bv_h) \Vert_{L^2(\Gamma_h)} + \Vert \bP \bv_h \Vert_{L^2(\Gamma_h)} \\
&\hspace*{4.8mm}\lesssim \Vert E_{T,h}(\bv_h) \Vert_{L^2(\Gamma_h)} + h^{k_g} \left( \Vert  \bv_h \Vert_{H^1(\Gamma_h)} + h^{-1} \Vert \bv_h \cdot \bn_h \Vert_{L^2(\Gamma_h)} \right) + \Vert \bP \bv_h \Vert_{L^2(\Gamma_h)}. 
%
\end{split}
\]
For $h$ sufficiently small the term $ h^{k_g}  \Vert  \bv_h \Vert_{H^1(\Gamma_h)}$ can be moved to the left hand side in \eqref{eqdiscreteKornsplit}. Hence, for the term $\Vert \bP \bv_h^l \Vert_{H^1(\Gamma)}$ we have a desired bound as in \eqref{discKorn1}. 
We now treat the normal component $\Vert  \bv_h^l \cdot \bn \Vert_{H^1(\Gamma)}$. Note that
\begin{equation} \label{eqnormalpart} \begin{split}
& \Vert  \bv_h^l \cdot \bn \Vert_{H^1(\Gamma)} 
\lesssim \Vert \bv_h^l \cdot \bn \Vert_{L^2(\Gamma)} + \Vert \nabla_\Gamma ( \bv_h^l\cdot \bn) \Vert_{L^2(\Gamma)} \\
 & \lesssim \Vert \bv_h \Vert_{L^2(\Gamma_h)} +\Vert \nabla_{\Gamma_h} ( \bv_h \cdot \bn) \Vert_{L^2(\Gamma_h)} 
 \lesssim \Vert \bv_h \Vert_{L^2(\Gamma_h)} +\Vert\bP_h (\nabla  \bv_h)^T \bn\Vert_{L^2(\Gamma_h)}.
\end{split}
\end{equation}
We introduce the linear parametric interpololation of $\bn$, $\hat \bn_h:= I_\Theta^1 \bn$. For this interpolation we have
\[
  \|\nabla \hat \bn_h \|_{L^\infty(\Omega_{\Theta}^{\Gamma})} \lesssim 1, ~~\|\hat \bn_h -  \bn \|_{L^\infty(\Omega_{\Theta}^{\Gamma})} \lesssim h, ~~
\|\hat \bn_h - \bn_h \|_{L^\infty(\Omega_{\Theta}^{\Gamma})} \lesssim h.
\]
Note that $\bv_h \cdot \hat\bn_h \in V_{h,\Theta}^{k+1}$. We obtain:
\begin{align*}
 \Vert \bP_h (\nabla  \bv_h)^T \bn\Vert_{L^2(\Gamma_h)} &\lesssim \Vert  (\nabla \bv_h)^T \hat \bn_h\Vert_{L^2(\Gamma_h)} + h \Vert  \nabla \bv_h \Vert_{L^2(\Gamma_h)} \\
 & \lesssim \Vert  \nabla (\bv_h \cdot \hat \bn_h)\Vert_{L^2(\Gamma_h)} + \Vert  \bv_h \Vert_{L^2(\Gamma_h)}+ h \Vert  \nabla \bv_h \Vert_{L^2(\Gamma_h)}.
\end{align*}
Using this in  \eqref{eqnormalpart} and applying the estimate \eqref{HH4a} yields
\begin{equation} \label{ee3}
  \Vert  \bv_h^l \cdot \bn \Vert_{H^1(\Gamma)} \lesssim \Vert \bv_h \Vert_{L^2(\Gamma_h)} + h^\frac12 \Vert \nabla \bv_h \bn_h\Vert_{L^2(\Omega_{\Theta}^{\Gamma})} + \Vert  \nabla (\bv_h \cdot \hat \bn_h)\Vert_{L^2(\Gamma_h)}.
\end{equation}
Using \eqref{HH4} and \eqref{HH2} we get
\begin{align*}
 &  \Vert  \nabla (\bv_h \cdot \hat \bn_h)\Vert_{L^2(\Gamma_h)}  \lesssim h^{-1} \Vert  \bv_h \cdot \hat \bn_h\Vert_{L^2(\Gamma_h)} + h^{-\frac12}
\Vert  \bn_h \cdot \nabla (\bv_h \cdot \hat \bn_h) \Vert_{L^2(\Omega_{\Theta}^{\Gamma})} \\
 & \lesssim \Vert  \bv_h \Vert_{L^2(\Gamma_h)} + h^{-1}  \Vert  \bv_h \cdot  \bn_h\Vert_{L^2(\Gamma_h)} + h^{-\frac12}
\Vert  \nabla \bv_h  \hat \bn_h \Vert_{L^2(\Omega_{\Theta}^{\Gamma})} + h^{-\frac12}
\Vert  \bv_h  \Vert_{L^2(\Omega_{\Theta}^{\Gamma})} \\
& \lesssim \Vert  \bv_h \Vert_{L^2(\Gamma_h)} + h^{-1}  \Vert  \bv_h \cdot  \bn_h\Vert_{L^2(\Gamma_h)} + h^{-\frac12}
\Vert  \nabla \bv_h   \bn_h \Vert_{L^2(\Omega_{\Theta}^{\Gamma})} + h^{-\frac12}
\Vert  \bv_h  \Vert_{L^2(\Omega_{\Theta}^{\Gamma})} \\
 & \lesssim \Vert  \bv_h \Vert_{L^2(\Gamma_h)} + h^{-1}\Vert  \bv_h \cdot \bn_h \Vert_{L^2(\Gamma_h)} + h^{-\frac12}
\Vert  \nabla \bv_h   \bn_h \Vert_{L^2(\Omega_{\Theta}^{\Gamma})}. 
\end{align*}
From this and \eqref{ee3} we get
\[ \begin{split}
 \Vert  \bv_h^l \cdot \bn \Vert_{H^1(\Gamma)}  & \lesssim \Vert \bv_h \Vert_{L^2(\Gamma_h)}+ h^{-1}\Vert  \bv_h \cdot \bn_h \Vert_{L^2(\Gamma_h)} + h^{-\frac12}
\Vert  \nabla \bv_h   \bn_h \Vert_{L^2(\Omega_{\Theta}^{\Gamma})} \\
 & \lesssim \Vert \bP_h\bv_h \Vert_{L^2(\Gamma_h)}+ h^{-1}\Vert  \bv_h \cdot \bn_h \Vert_{L^2(\Gamma_h)} + h^{-\frac12}
\Vert  \nabla \bv_h   \bn_h \Vert_{L^2(\Omega_{\Theta}^{\Gamma})},
 \end{split}
\]
hence, also for  the normal part we have a bound as in \eqref{discKorn1}, which completes the proof.
\end{proof}

\begin{corollary} \label{corolKorn}
For $h$ sufficiently small and for arbitrary $\bv_h \in \bV_{h,\Theta}^k$ we have
 \begin{align}
\Vert \bv_h \Vert_{H^1(\Gamma_h)} &\lesssim \Vert E_{T,h}(\bv_h) \Vert_{L^2(\Gamma_h)} + \Vert \bP_h \bv_h \Vert_{L^2(\Gamma_h)} + h^{-1} \Vert \bv_h \cdot \tilde{\bn}_h \Vert_{L^2(\Gamma_h)} \label{discKorn1a} \\
&\hspace*{4.8mm}+ h^{-\frac{1}{2}} \Vert \nabla \bv_h \bn_h \Vert_{L^2(\Omega_\Theta^\Gamma)},  \nonumber\\
\Vert \bv_h \Vert_{H^1(\Gamma_h)} &\lesssim \Vert E_{h}(\bv_h) \Vert_{L^2(\Gamma_h)} + \Vert \bv_h \Vert_{L^2(\Gamma_h)} + h^{-1} \Vert \bv_h \cdot \tilde{\bn}_h \Vert_{L^2(\Gamma_h)} \label{discKorn2}\\
&\hspace*{4.8mm}+ h^{-\frac{1}{2}} \Vert \nabla \bv_h \bn_h \Vert_{L^2(\Omega_\Theta^\Gamma)}. \nonumber
\end{align}
\end{corollary}
\begin{proof} 
Note that $\|\bv_h \cdot \bn_h\|_{L^2(\Gamma_h)} \lesssim \|\bv_h \cdot \tilde \bn_h\|_{L^2(\Gamma_h)} + h^{k_g}\|\bv_h\|_{L^2(\Gamma_h)}$. Hence, the result \eqref{discKorn1a} is a consequence of \eqref{discKorn1}.
Using the definitions of $E_{T,h}(\cdot)$, $E_h(\cdot)$ and a triangle inequality the result \eqref{discKorn2} immediately follows from \eqref{discKorn1a}. 
 \end{proof}
\ \\
\begin{remark} \rm 
From the proof one can see that in the estimate \eqref{discKorn1} the part $\Vert E_{T,h}(\bv_h) \Vert_{L^2(\Gamma_h)} + \Vert \bP_h \bv_h \Vert_{L^2(\Gamma_h)}$ is the key term  to bound the $H^1(\Gamma_h)$ norm of the tangential component of the vector function $\bv_h$, and the part  $h^{-1} \Vert \bv_h \cdot {\bn}_h \Vert_{L^2(\Gamma_h)}+ h^{-\frac{1}{2}} \Vert \nabla \bv_h \bn_h \Vert_{L^2(\Omega_\Theta^\Gamma)}$ is essential to bound the normal component.
\end{remark}

\subsubsection{Consistency error of the penalty methods \eqref{discretepenaltyform1} and \eqref{discretepenaltyform2}} \label{sectcons1}
Based on the results obtained in the previous sections the derivation of satisfactory consistency error bounds is straightforward. In this section we derive these bounds for the two penalty methods. Using the definitions of the bilinear forms $A_h^{P_i}(\cdot,\cdot)$, $i=1,2$, we obtain from Corollary~\ref{corolKorn}, for $\rho \sim h^{-1}$ and $\eta \gtrsim h^{-2}$:
\begin{align} \label{K1}
 \|\bv_h \|_{H^1(\Gamma_h)}^2 \lesssim A_h^{P_i}(\bv_h, \bv_h) , \quad \text{for all}~~\bv_h \in \bV_{h,\Theta}^k.
\end{align}

\begin{lemma} \label{lemmaconsistencyerror}
Let $\bu = \bu_T\in \bV_T$ be the unique solution of problem \eqref{contform}. We  assume that the data error satisfies $\Vert \vert \bB\vert \bbf^e - \bbf_h \Vert_{L^2(\Gamma_h)} \lesssim h^{k_g+1} \Vert \bbf \Vert_{L^2(\Gamma)}$ and $\rho \sim h^{-1}$, $\eta \gtrsim h^{-2}$. Then the following bounds hold
\begin{align}
\sup_{\bw_h \in \bV_{h,\Theta}^{k}} \frac{\vert A_h^{P_1}(\bu^e, \bw_h) - (\bbf_h, \bw_h)_{L^2(\Gamma_h)} \vert}{\Vert \bw_h \Vert_{A_h^{P_1}}} &\lesssim (h^{k_g} + \eta^{\frac{1}{2}}h^{k_p} + \eta^{-\frac{1}{2}}) \Vert \bu \Vert_{H^1(\Gamma)}, \label{eqconsp1}\\ 
\sup_{\bw_h \in \bV_{h,\Theta}^{k}} \frac{\vert A_h^{P_2}(\bu^e, \bw_h) - (\bbf_h, \bw_h)_{L^2(\Gamma_h)} \vert}{\Vert \bw_h \Vert_{A_h^{P_2}}} &\lesssim (h^{k_g} + \eta^{\frac{1}{2}}h^{k_p}) \Vert \bu \Vert_{H^1(\Gamma)} \label{eqconsp2}.
\end{align}
\end{lemma}
\begin{proof}
We start with \eqref{eqconsp2}. Take $\bw_h \in \bV_{h,\Theta}^{k}$. We have, cf. \eqref{consterm2},
\begin{equation*} \begin{split}
A_h^{P_2}(\bu^e, \bw_h) - (\bbf_h, \bw_h)_{L^2(\Gamma_h)}= G_{a_T}(\bu^e, \bw_h) + s_h(\bu^e,\bw_h) + k_h(\bu^e,\bw_h) + G_f(\bw_h).
\end{split}
\end{equation*}
Using Lemma \ref{lemmageometryapproximationerrors}, \eqref{K1} and $\Vert \bu^e \cdot \bn_h \Vert_{L^2(\Gamma_h)}= \Vert \bu^e \cdot (\bn_h- \bn) \Vert_{L^2(\Gamma_h)}\lesssim h^{k_g} \|\bu^e\|_{L^2(\Gamma_h)}$,   we get 
\begin{align} 
&\vert G_{a_T}(\bu^e, \bw_h) \vert \nonumber \\
&\hspace*{4.8mm}\lesssim h^{k_g} \left(\Vert \bu^e \Vert_{H^1(\Gamma_h)} + h^{-1} \Vert \bu^e \cdot \bn_h \Vert_{L^2(\Gamma_h)}\right) \left(\Vert \bw_h \Vert_{H^1(\Gamma_h)} + h^{-1} \Vert \bw_h \cdot \bn_h \Vert_{L^2(\Gamma_h)}\right) \nonumber \\
&\hspace*{4.8mm}\lesssim h^{k_g} \Vert \bu \Vert_{H^1(\Gamma)} 
  \left(\Vert \bw_h \Vert_{H^1(\Gamma_h)} +  h^{-1} \Vert \bw_h \cdot \tilde{\bn}_h \Vert_{L^2(\Gamma_h)} \right) \nonumber \\
&\hspace*{4.8mm}\lesssim h^{k_g} \Vert \bu \Vert_{H^1(\Gamma)} \Vert \bw_h \Vert_{A_h^{P_2}}. \label{eqgat}
\end{align}
We also have 
\begin{equation} \label{eqgf} \begin{split}
\vert G_f(\bw_h) \vert &\lesssim h^{k_g+1} \Vert \bbf \Vert_{L^2(\Gamma)} \Vert \bw_h \Vert_{L^2(\Gamma_h)} 
\lesssim h^{k_g+1} \Vert \bu \Vert_{H^1(\Gamma)} \Vert \bw_h \Vert_{A_h^{P_2}}.
\end{split}
\end{equation}
Using inequality \eqref{lemmasobolevnormsneighborhood} we obtain
\begin{equation} \label{eqsh} \begin{split}
& \vert s_h(\bu^e,\bw_h) \vert \lesssim h^{-1} \Vert \nabla \bu^e \bn_h \Vert_{L^2(\Omega_\Theta^\Gamma)} \Vert \nabla \bw_h \bn_h \Vert_{L^2(\Omega_\Theta^\Gamma)} \\
& \lesssim h^{-1} \Vert \nabla \bu^e (\bn_h - \bn) \Vert_{L^2(\Omega_\Theta^\Gamma)} \Vert \nabla \bw_h \bn_h \Vert_{L^2(\Omega_\Theta^\Gamma)} \\
&\lesssim h^{k_g-1} \Vert \nabla \bu^e \Vert_{L^2(\Omega_\Theta^\Gamma)} \Vert \nabla \bw_h \bn_h \Vert_{L^2(\Omega_\Theta^\Gamma)} \\
&\lesssim h^{k_g-\frac{1}{2}} \Vert \bu \Vert_{H^1(\Gamma)} \Vert \nabla \bw_h \bn_h \Vert_{L^2(\Omega_\Theta^\Gamma)} 
\lesssim h^{k_g} \Vert \bu \Vert_{H^1(\Gamma)} \Vert \bw_h \Vert_{A_h^{P_2}}.
\end{split}
\end{equation}
The penalty term can be estimated as follows:
\begin{equation*} \begin{split}
& \vert k_h(\bu^e,\bw_h) \vert \lesssim \eta \Vert \bu^e \cdot \tilde{\bn}_h \Vert_{L^2(\Gamma_h)} \Vert \bw_h \cdot \tilde{\bn}_h \Vert_{L^2(\Gamma_h)} \\
&\lesssim \eta^{\frac{1}{2}} \Vert \bu^e \cdot (\tilde{\bn}_h - \bn) \Vert_{L^2(\Gamma_h)} \Vert \bw_h \Vert_{A_h^{P_2}} 
\lesssim \eta^{\frac{1}{2}}h^{k_p} \Vert \bu \Vert_{L^2(\Gamma)} \Vert \bw_h \Vert_{A_h^{P_2}}. 
\end{split}
\end{equation*}
Combining these estimates completes the proof for \eqref{eqconsp2}. Next we show \eqref{eqconsp1}. Recall that, cf.~\eqref{consterm1},
\begin{equation*} \begin{split}
&A_h^{P_1}(\bu^e, \bw_h) - (\bbf_h, \bw_h)_{L^2(\Gamma_h)} \\
& = G_{a}(\bu^e, \bw_h) + ( E(\bu) , E((\bw_h^l \cdot \bn)\bn))_{L^2(\Gamma)} + s_h(\bu^e,\bw_h) + k_h(\bu^e,\bw_h) + G_f(\bw_h).
\end{split}
\end{equation*}
The terms $G_{a}(\cdot,\cdot)$, $s_h(\cdot,\cdot)$, $k_h(\cdot,\cdot)$ and $G_f(\cdot)$ can be estimated as above. We treat the remaining term. Note that (cf.~\eqref{identi}), 
$
E((\bw_h^l \cdot \bn)\bn) = (\bw_h^l \cdot \bn)\bH$.
Using the lemmas \ref{lemmaB} and \ref{lemmanormequivalences} we get
\begin{align} 
&\vert ( E(\bu) , (\bw_h^l \cdot \bn)\bH)_{L^2(\Gamma)} \vert \nonumber  \\ & = \vert  ( \vert \bB \vert (E(\bu))^e , (\bw_h \cdot \bn)\bH)_{L^2(\Gamma_h)} \vert
\lesssim \Vert \bu^e \Vert_{H^1(\Gamma_h)} \Vert \bw_h \cdot \bn \Vert_{L^2(\Gamma_h)} \nonumber \\
&\lesssim \Vert \bu \Vert_{H^1(\Gamma)} \Vert \bw_h \cdot \bn \Vert_{L^2(\Gamma_h)} 
\lesssim \Vert \bu \Vert_{H^1(\Gamma)} (\Vert \bw_h \cdot (\bn - \tilde{\bn}_h) \Vert_{L^2(\Gamma_h)} + \Vert \bw_h \cdot \tilde{\bn}_h \Vert_{L^2(\Gamma_h)})   \nonumber \\
&\lesssim \Vert \bu \Vert_{H^1(\Gamma)} (h^{k_p}\Vert \bw_h \Vert_{L^2(\Gamma_h)} + \Vert \bw_h \cdot \tilde{\bn}_h \Vert_{L^2(\Gamma_h)}) \nonumber \\
& \lesssim  (h^{k_p} + \eta^{-\frac12})\Vert \bu \Vert_{H^1(\Gamma)} \Vert \bw_h \Vert_{A_h^{P_1}}. \label{consA}
\end{align}
Since $k_p \geq k_g$ we get \eqref{eqconsp1}.
\end{proof}

In Lemma~\ref{lemmaconsistencyerror}, for the stability and penalty parameters we restrict to $\rho \sim  h^{-1}$ and $\eta \gtrsim h^{-2}$. For these parameter values we then have  the estimate \eqref{K1}, which is used at several places in the proof. 

\subsubsection{Consistency error of the Lagrange Method \eqref{discretelagrangeform}} \label{sectcons2}
In this section we derive bounds for the Lagrange multiplier method. For this method we do not have an analog of \eqref{K1} of the form $\|\bv_h \|_{H^1(\Gamma_h)}^2 \lesssim A_h^{L}(\bv_h, \bv_h)$ for all $\bv_h \in \bV_{h,\Theta}^k$. Such an estimate is problematic, because the term $h^{-1} \|\bv_h \cdot \bn_h\|_{L^2(\Gamma_h)}$ that occurs in the discrete Korn's type inequality \eqref{discKorn1} can not be controlled by the bilinear form $A_h^L(\cdot,\cdot)$. Instead we only have the (weaker) bound
\begin{equation} \label{K1a}
 \|\bv_h \|_{H^1(\Gamma_h)}^2 \lesssim  h^{-2} A_h^{L}(\bv_h, \bv_h) \quad \text{for all}~~\bv_h \in \bV_{h,\Theta}^k,
\end{equation}
which follows from \eqref{HH4a} and the definition of $A_h^L(\cdot,\cdot)$, cf. Remark~\ref{rembound}.
\begin{lemma} \label{lemmaconsistencyerrorlagrange}
Let $(\bu, \lambda) \in \bV_* \times L^2(\Gamma)$ be the unique solution of problem \eqref{contformlagrange}. We further assume that the data error satisfies $\Vert \vert \bB\vert \bbf^e - \bbf_h \Vert_{L^2(\Gamma_h)} \lesssim h^{k_g+1} \Vert \bbf \Vert_{L^2(\Gamma)}$ and Assumption \ref{assumptionrhoL} holds. Then we obtain the following bound:
\begin{equation*} \begin{split}
&\sup_{(\bw_h,\mu_h)  \in \bV_{h,\Theta}^{k} \times V_{h,\Theta}^{k_l}} \frac{\vert \mathcal{A}_h((\bu^e,\lambda^e), (\bw_h,\mu_h)) - (\bbf_h, \bw_h)_{L^2(\Gamma_h)} \vert}{\left(\Vert \bw_h \Vert_{A_h^L}^2 + \Vert \mu_h \Vert_{M}^2\right)^\frac{1}{2}} \\
&\hspace*{4.8mm}\lesssim h^{k_g-1}  \Vert \bu \Vert_{H^1(\Gamma)} + h^{k_g} \Vert \lambda \Vert_{H^1(\Gamma)}.
\end{split}
\end{equation*}
\end{lemma}
\begin{proof}
Take $(\bw_h,\mu_h)  \in \bV_{h,\Theta}^{k} \times V_{h,\Theta}^{k_l}$. 
Using \eqref{consterm3} we obtain
\begin{equation*} \begin{split}
&\mathcal{A}_h((\bu^e,\lambda^e), (\bw_h,\mu_h)) - (\bbf_h, \bw_h)_{L^2(\Gamma_h)} \\
& =G_a(\bu^e, \bw_h) + b_h(\bw_h,\lambda^e) + b_h(\bu^e, \mu_h) +  s_h(\bu^e,\bw_h) - (\bw_h^l \cdot \bn, \lambda)_{L^2(\Gamma)} + G_f(\bw_h)\\
&\hspace*{4.8mm}= \underbrace{G_a(\bu^e, \bw_h)}_{(1)} + \underbrace{(\bw_h \cdot \bn_h,\lambda^e)_{L^2(\Gamma_h)} - (\bw_h^l \cdot \bn, \lambda)_{L^2(\Gamma)}}_{(2)} + \underbrace{(\bu^e \cdot \bn_h,\mu_h)_{L^2(\Gamma_h)}}_{(3)} \\
&\hspace*{9.6mm}+ \underbrace{s_h(\bu^e,\bw_h)}_{(4)} + \underbrace{\tilde{s}_h(\bw_h,\lambda^e)}_{(5)} + \underbrace{\tilde{s}_h(\bu^e,\mu_h)}_{(6)} + \underbrace{G_f(\bw_h)}_{(7)}.
\end{split}
\end{equation*}
We derive bounds for these seven terms. We start with term $(1)$. Applying Lemma \ref{lemmageometryapproximationerrors} and \eqref{K1a} we get 
\begin{equation*} 
\vert G_a(\bu^e, \bw_h)  \vert \lesssim h^{k_g} \Vert \bu^e \Vert_{H^1(\Gamma_h)} \Vert \bw_h \Vert_{H^1(\Gamma_h)} 
\lesssim h^{k_g-1} \Vert \bu \Vert_{H^1(\Gamma)} \Vert \bw_h \Vert_{A_h^L}.
\end{equation*}
With  Lemma~\ref{lemmaB}  we obtain for term $(2)$
\begin{equation*} \begin{split}
&\vert (\bw_h \cdot \bn_h,\lambda^e)_{L^2(\Gamma_h)} - (\bw_h^l \cdot \bn, \lambda)_{L^2(\Gamma)} \vert \\
&= \vert (\bw_h \cdot \bn_h,\lambda^e)_{L^2(\Gamma_h)} - (\vert \bB \vert (\bw_h \cdot \bn), \lambda^e)_{L^2(\Gamma_h)} \vert  \\
&= \vert (\bw_h \cdot (\bn_h - \bn),\lambda^e)_{L^2(\Gamma_h)} - ((\vert \bB \vert - 1) (\bw_h \cdot \bn), \lambda^e)_{L^2(\Gamma_h)} \vert  
\lesssim  h^{k_g} \Vert \lambda \Vert_{L^2(\Gamma)} \Vert \bw_h \Vert_{A_h^L}. 
\end{split}
\end{equation*}
For term $(3)$ we have
\begin{equation*} \begin{split}
\vert (\bu^e \cdot \bn_h,\mu_h)_{L^2(\Gamma_h)} \vert &= \vert (\bu^e \cdot (\bn_h - \bn),\mu_h)_{L^2(\Gamma_h)}\vert \\
&\lesssim \Vert \bn_h - \bn \Vert_{L^\infty(\Gamma_h)} \Vert \bu^e \Vert_{L^2(\Gamma_h)} \Vert \mu_h \Vert_{L^2(\Gamma_h)} 
\lesssim h^{k_g} \Vert \bu \Vert_{L^2(\Gamma)} \Vert \mu_h \Vert_{M}.
\end{split}
\end{equation*}
The terms $(4)$ and $(7)$ can be estimated as in \eqref{eqsh} and \eqref{eqgf}:
\begin{equation*}
\vert s_h(\bu^e,\bw_h) \vert \lesssim h^{k_g} \Vert \bu \Vert_{H^1(\Gamma)} \Vert \bw_h \Vert_{A_h^L}, 
~~\vert G_f(\bw_h) \vert \lesssim h^{k_g+1} \Vert \bu \Vert_{H^1(\Gamma)} \Vert \bw_h \Vert_{A_h^L}.
\end{equation*}
Finally, for the terms $(5)$, $(6)$ we can apply arguments as in \eqref{eqsh}, resulting in
\begin{equation*}
\vert \tilde{s}_h(\bw_h,\lambda^e) \vert \lesssim  h^{k_g} \Vert \lambda \Vert_{H^1(\Gamma)} \Vert \bw_h \Vert_{A_h^L}, ~~
\vert \tilde{s}_h(\bu^e,\mu_h) \vert \lesssim  h^{k_g} \Vert \bu \Vert_{H^1(\Gamma)} \Vert \mu_h \Vert_{M}.
\end{equation*} 
Combining the bounds for these terms which completes the proof.
\end{proof}

Note that compared to the consistency error bounds for the penalty methods in Lemma~\ref{lemmaconsistencyerror}, for the Lagrange multiplier method we (only) have $h^{k_g-1} \|\bu\|_{H^1(\Gamma)}$ (instead of  $h^{k_g} \|\bu\|_{H^1(\Gamma)}$). The loss of one power in $h$ is caused by the estimate \eqref{K1a}.

\subsection{Discretization error bounds} \label{sectDiscrerror}

We combine the Strang-Lemma~\ref{stranglemma} and the bounds for the approximation error and the consistency error to obtain bounds for the discretization error in the energy norms. We first consider the inconsistent penalty formulation  \eqref{discretepenaltyform1}.

\begin{theorem} \label{energyerrorboundph1}
Let $\bu \in \bV$ and   $\bu_h \in \bV_{h,\Theta}^{k}$ be the solution of  \eqref{contform} and of \eqref{discretepenaltyform1}, respectively.  We  assume that the data error satisfies $\Vert \vert \bB\vert \bbf^e - \bbf_h \Vert_{L^2(\Gamma_h)} \lesssim h^{k_g+1} \Vert \bbf \Vert_{L^2(\Gamma)}$ and $\rho \sim h^{-1}$, $\eta \gtrsim h^{-2}$. Then the following bound holds
\begin{equation} \label{bound1}
\Vert \bu^e -\bu_h \Vert_{A_h^{P_1}} \lesssim (h^{k}   + \eta^{\frac{1}{2}} h^{k+1}) \Vert \bu \Vert_{H^{k+1}(\Gamma)} + (h^{k_g} + \eta^{\frac{1}{2}}h^{k_p} + \eta^{-\frac{1}{2}}) \Vert \bu \Vert_{H^1(\Gamma)}.
\end{equation}
\end{theorem}

\begin{remark}\label{remdiscussion} \rm We discuss this error bound. 
For \emph{linear} finite elements, i.e. $k =1$,   $k_g = 1$ (linear geometry approximation),  $k_p =2$ (higher order normal approximation in the penalty term) and $\eta \sim h^{-2}$  we obtain an optimal order error bound. However, for \emph{higher order} finite elements, i.e. $k \geq 2$, we are not able to choose the other parameters ($k_g,k_p,\eta$) such that we have an optimal order error bound.  If we balance the terms $\eta^\frac12 h^{k+1}$ and $\eta^{-\frac12}$ this yields $\eta \sim h^{-(k+1)}$. Using this parameter choice and $k_g=k$ (isoparametric case), $k_p=k+2$ (higher order normal approximation in the penalty term), we obtain an (suboptimal) error bound of the order $h^{\frac12 (k+1)}$. This suboptimal result is due to the factor $\eta^{-\frac12}$ in the error bound, which is caused (only) by the estimate for the inconstency term $( E(\bu) , (\bw_h^l \cdot \bn)\bH)_{L^2(\Gamma)}  $ in \eqref{consA}. 
\end{remark}

Next we consider the consistent penalty formulation \eqref{discretepenaltyform2}. 

\begin{theorem} \label{energyerrorboundph2}
Let $\bu \in \bV$ and  $\bu_h \in \bV_{h,\Theta}^{k}$ be the  solution of \eqref{contform} and  of  \eqref{discretepenaltyform2}, respectively. We assume that the data error satisfies $\Vert \vert \bB\vert \bbf^e - \bbf_h \Vert_{L^2(\Gamma_h)} \lesssim h^{k_g+1} \Vert \bbf \Vert_{L^2(\Gamma)}$ and $\rho \sim h^{-1}$, $\eta \gtrsim h^{-2}$. Then the following bound holds:
\begin{equation} \label{bound2}
\Vert \bu^e -\bu_h \Vert_{A_h^{P_2}} \lesssim (h^{k}  + \eta^{\frac{1}{2}} h^{k+1}) \Vert \bu \Vert_{H^{k+1}(\Gamma)} + (h^{k_g} + \eta^{\frac{1}{2}}h^{k_p}) \Vert \bu \Vert_{H^1(\Gamma)}.
\end{equation}
\end{theorem}

\begin{remark} \label{remdisc2} \rm Note that the bound in \eqref{bound2} is the same as in \eqref{bound1}, except for the term $\eta^{-\frac12}$ that occurs in \eqref{bound1} due to the inconsistency of the method \eqref{discretepenaltyform1}. 
In view of the factor  $\eta^{\frac{1}{2}} h^{k+1}$  we take  $\eta \sim h^{-2}$. Based on the consistency error term we take  $k_g = k$ (isoparametric case)  and  $k_p=k+1$ (higher order normal approximation in the penalty term). This then yields an optimal order error bound.
\end{remark}

The same estimates as in \eqref{bound1} and \eqref{bound2} also hold with the energy norm $\|\cdot\|_{A_h^{P_i}}$ replaced by the $H^1(\Gamma_h)$ norm:
\begin{corollary}
Let $\bu \in \bV$,  $\bu_h \in \bV_{h,\Theta}^{k}$, $\tilde{\bu}_h \in \bV_{h,\Theta}^{k}$ be  solution of  \eqref{contform},  \eqref{discretepenaltyform1} and of \eqref{discretepenaltyform2}, respectively. The  following discretization errror bounds hold:
\begin{align*}
\Vert \bu^e -\bu_h \Vert_{H^1(\Gamma_h)} &\lesssim (h^{k} +  \eta^{\frac{1}{2}} h^{k+1}) \Vert \bu \Vert_{H^{k+1}(\Gamma)} + (h^{k_g} + \eta^{\frac{1}{2}}h^{k_p} + \eta^{-\frac{1}{2}}) \Vert \bu \Vert_{H^1(\Gamma)}, \\
\Vert \bu^e -\tilde{\bu}_h \Vert_{H^1(\Gamma_h)} &\lesssim (h^{k} +  \eta^{\frac{1}{2}} h^{k+1}) \Vert \bu \Vert_{H^{k+1}(\Gamma)} + (h^{k_g} + \eta^{\frac{1}{2}}h^{k_p}) \Vert \bu \Vert_{H^1(\Gamma)}.
\end{align*}
\end{corollary} 
\begin{proof}
We show the first bound. The second bound can be shown analogously. Using Lemma \ref{lemmascalarapproximationerror} and inequality \eqref{K1} we get
\begin{align*}
\Vert \bu^e -\bu_h \Vert_{H^1(\Gamma_h)} &\leq \Vert \bu^e - I^{k}_{\Theta} (\bu^e) \Vert_{H^1(\Gamma_h)} + \Vert I^{k}_{\Theta} (\bu^e) - \bu_h \Vert_{H^1(\Gamma_h)} \\
&\lesssim h^k \Vert \bu \Vert_{H^{k+1}(\Gamma)} + \Vert I^{k}_{\Theta} (\bu^e) - \bu_h \Vert_{A_h^{P_1}} . 
\end{align*}
Since 
\begin{align*}
\Vert I^{k}_{\Theta} (\bu^e) - \bu_h \Vert_{A_h^{P_1}} \leq \Vert I^{k}_{\Theta} (\bu^e) - \bu^e \Vert_{A_h^{P_1}} + \Vert \bu^e - \bu_h \Vert_{A_h^{P_1}}
\end{align*}
we get the desired result using Lemma \ref{lemmaapproximationerror} and Theorem \ref{energyerrorboundph2}
\end{proof}

Finally we consider the Lagrange multiplier formulation. 

\begin{theorem} \label{energyerrorboundlh}
Let $(\bu, \lambda) \in \bV_* \times L^2(\Gamma)$ and  $(\bu_h, \lambda_h) \in V_{h,\Theta}^{k} \times V_{h,\Theta}^{k_l}$ be the solution of \eqref{contformlagrange} and of  \eqref{discretelagrangeform}, respectively. We   assume that the data error satisfies $\Vert \vert \bB\vert \bbf^e - \bbf_h \Vert_{L^2(\Gamma_h)} \lesssim h^{k_g+1} \Vert \bbf \Vert_{L^2(\Gamma)}$ and Assumption \ref{assumptionrhoL} holds. Then we obtain the following error bound:
\begin{equation*} \begin{split}
&\Vert \bu^e - \bu_{h} \Vert_{A_h^L} + \Vert \lambda^e - \lambda_{h} \Vert_{M} \\
&\lesssim  h^{k}  \Vert \bu \Vert_{H^{k+1}(\Gamma)} + (h^{k_l+1} +\rho^{\frac{1}{2}} h^{k_l+\frac{1}{2}}) \Vert \lambda \Vert_{H^{k_l+1}(\Gamma)} 
+ h^{k_g-1}  \Vert \bu \Vert_{H^1(\Gamma)} + h^{k_g}  \Vert \lambda \Vert_{H^1(\Gamma)}
\end{split}
\end{equation*}
\end{theorem}

\begin{remark} \label{Remdisc2} \rm 
In the case of isoparametric finite elements, i.e. $k = k_g$, we do \emph{not} get an optimal order error bound. For the case of superparametric finite elements, i.e. $k_g = k+1$, we  distinguish   two cases. First, for $k_l = k$ (same degree finite elements for the Lagrange multiplier as for the primal variable) we can take any  $\rho=c_\alpha h^{1-\alpha}$, $\alpha \in [0,2]$, and $c_\alpha$ as in Corollary \ref{corollarydiscreteinfsup}.  For $k_l= k-1$ ($k \geq 2$) we  restrict to $\rho = c_\alpha h$ with $c_\alpha$ as in Corollary \ref{corollarydiscreteinfsup}. In both cases we then obtain an optimal order error  bound.
\end{remark}

\section{Numerical experiments} \label{sectExperiments}

In this section we present results of a few numerical experiments. We implemented the different methods in Netgen/NGSolve with ngsxfem \cite{ngsolve,ngsxfem}. 

For $\Gamma$ we take the unit sphere which is characterized by the zero level of the distance function function $\phi(x) = \sqrt{x_1^2+x_2^2+x_3^2} - 1$, $x= (x_1,x_2,x_3)^T$.
The surface is embedded in the domain $\Omega = [-1.5,1.5]^3$. We start with an unstructured tetrahedral Netgen-mesh with $h_{max} = 0.5$ (see \cite{Schoeberl1997}) and locally refine the mesh  using a marked-edge bisection method (refinement of tetrahedra that are intersected by the surface). We consider the vector Laplace problem \eqref{contform} with the prescribed solution
\begin{equation*}
\bu^*(x) = \bP(x)\left(-\frac{x_3^2}{x_1^2+x_2^2+x_3^2}, \frac{x_2}{\sqrt{x_1^2+x_2^2+x_3^2}}, \frac{x_1}{\sqrt{x_1^2+x_2^2+x_3^2}}\right)^T.
\end{equation*}
The solution is tangential, i.e. $\bP\bu^* = \bu^*$, and constant in normal direction, i.e. $\bu^* = (\bu^*)^e$. The right-hand side $\bbf$ is computed according to equation \eqref{eqstrong}.

We first consider the penalty formulations \eqref{discretepenaltyform1} and \eqref{discretepenaltyform2}.  The normal approximation $\tilde{\bn}_h$ used in the penalty term is computed as follows. We interpolate the exact level set function $\phi$ in the finite element space $V_{h,\Theta}^{k_p}$, which we denote by $\tilde{\phi}_h$, and then set $\tilde{\bn}_h := \frac{\nabla \tilde{\phi}_h}{\Vert \nabla \tilde{\phi}_h \Vert_2}$. For the approximation of the Weingarten mapping (needed only in \eqref{discretepenaltyform2}) we take $\bH_h = \nabla(I_{\Theta}^{k_g}(\bn_h))$. The resulting linear systems are solved using a direct solver.

We start with problem \eqref{discretepenaltyform1}. In Figure \ref{fig:P1h} the error measured in the $\Vert \cdot \Vert_{A_h^{P_1}}$-norm is shown, for different choices of parameters and refinement levels. 
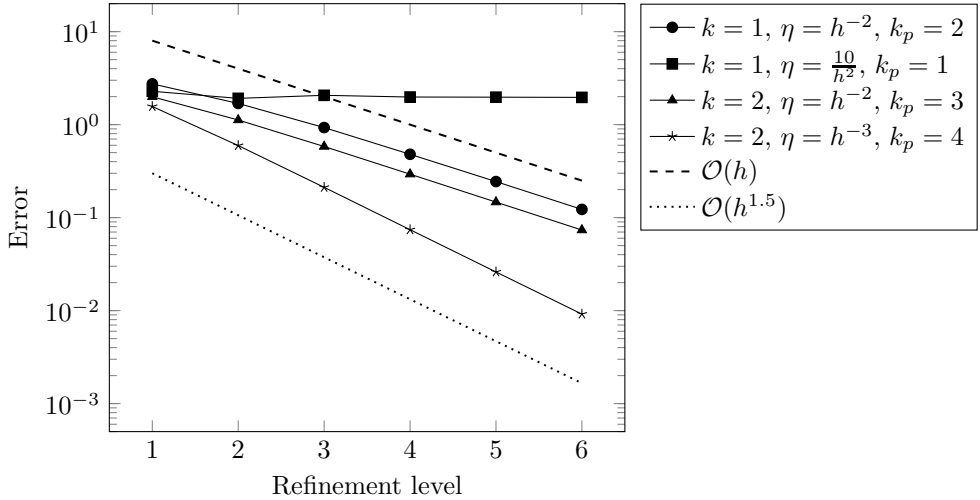
\begin{figure}[ht!]
  \begin{tikzpicture}
  \def\vara{8.0}
  \def\varb{0.3}
  \begin{semilogyaxis}[ xlabel={Refinement level}, ylabel={Error}, ymin=5E-4, ymax=20, legend style={ cells={anchor=west}, legend pos=outer north east}, cycle list name=mark list ]
    \addplot table[x=level, y=uerror1] {P1h.dat};
    \addplot table[x=level, y=uerror2] {P1h.dat};
    \addplot table[x=level, y=uerror3] {P1h.dat};
    \addplot table[x=level, y=uerror4] {P1h.dat};
    \addplot[dashed,line width=0.75pt] coordinates { 
    (1,\vara) (2,\vara*0.5) (3,\vara*0.25) (4,\vara*0.125) (5,\vara*0.0625)(6,\vara*0.03125)
    };
    \addplot[dotted,line width=0.75pt] coordinates { 
      (1,\varb) (2,\varb*0.5*0.707106781) (3,\varb*0.25*0.5) (4,\varb*0.125*0.353553391) (5,\varb*0.0625*0.25)(6,\varb*0.03125*0.176776695)
    };
    \legend{$k=1\text{, }\eta = h^{-2}\text{, }k_p=2$, $k=1\text{, }\eta = \frac{10}{h^{2}}\text{, }k_p=1$,$k=2\text{, }\eta = h^{-2}\text{, }k_p=3$, $k=2\text{, }\eta = h^{-3}\text{, }k_p=4$, $\mathcal{O}(h)$, $\mathcal{O}(h^{1.5})$}
  \end{semilogyaxis}
  \end{tikzpicture}
  \caption{$\Vert \cdot \Vert_{A_h^{P_1}}$-error for problem \eqref{discretepenaltyform1} with $k_g = k$ and $\rho = h^{-1}$.}
  \label{fig:P1h}
\end{figure}

For isoparametric linear finite elements ($k=k_g=1$) and a one order higher normal approximation for the penalty term ($k_p=2$) we observe optimal $\mathcal{O}(h)$-convergence. Choosing the same order for the normal approximation ($k_p=1$)  we do not have convergence. 
In the experiment with $k=k_g=k_p=1$ we used $\eta = 10 \cdot h^{-2}$ (instead of $\eta =  h^{-2}$) to have a bigger constant in the term $\eta^\frac12 h^{k_p}$ in \eqref{bound1},  in order to see the loss of one order more clearly. For the case $k=2$ we do not observe optimal (second order) convergence in Figure \ref{fig:P1h}. For $k=2$, $\eta=h^{-2}$, $k_p=3$ we obtain (only) first order convergence, whereas for $k=2$, $\eta=h^{-3}$, $k_p=4$ the error behaves as $\sim h^{1.5}$. 
 All these results are in agreement with the bounds in Theorem~\ref{energyerrorboundph1}, cf. Remark~\ref{remdiscussion}.

Next we consider problem \eqref{discretepenaltyform2}. In Figure \ref{fig:P2h} we show the discretization error measured in the $\Vert \cdot \Vert_{A_h^{P_2}}$-norm for different choices of parameters and refinement levels. For isoparametric finite elements ($k = k_g$) and a one order higher normal approximation for the penalty term ($k_p = k + 1$) we observe optimal $\mathcal{O}(h^{k})$-convergence for $k=1,\ldots, 3$. For isoparametric quadratic finite elements ($k=k_g=2$) and a normal approximation of order two in the penalty term ($k_p=2$) we observe a loss of one order, i.e. $\mathcal{O}(h)$-convergence.  All these results are in agreement with the bounds in Theorem~\ref{energyerrorboundph2}, cf. Remark~\ref{remdisc2}.

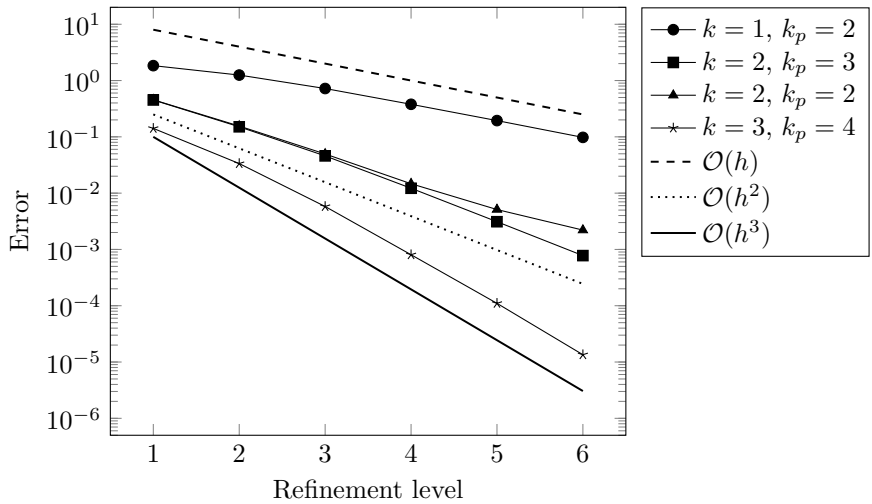
\begin{figure}[ht!]
  \begin{tikzpicture}
  \def\vara{8.0}
  \def\varb{0.25}
  \def\varc{0.1}
  \begin{semilogyaxis}[ xlabel={Refinement level}, ylabel={Error}, ymin=5E-7, ymax=20, legend style={ cells={anchor=west}, legend pos=outer north east}, cycle list name=mark list ]
    \addplot table[x=level, y=uerror1] {P2h.dat};
    \addplot table[x=level, y=uerror2] {P2h.dat};
    \addplot table[x=level, y=uerror3] {P2h.dat};
    \addplot table[x=level, y=uerror4] {P2h.dat};
    \addplot[dashed,line width=0.75pt] coordinates { 
    (1,\vara) (2,\vara*0.5) (3,\vara*0.25) (4,\vara*0.125) (5,\vara*0.0625)(6,\vara*0.03125)
    };
    \addplot[dotted,line width=0.75pt] coordinates { 
      (1,\varb) (2,\varb*0.5*0.5) (3,\varb*0.25*0.25) (4,\varb*0.125*0.125) (5,\varb*0.0625*0.0625)(6,\varb*0.03125*0.03125)
    };
    \addplot[solid,line width=0.75pt] coordinates { 
      (1,\varc) (2,\varc*0.5*0.5*0.5) (3,\varc*0.25*0.25*0.25) (4,\varc*0.125*0.125*0.125) (5,\varc*0.0625*0.0625*0.0625)(6,\varc*0.03125*0.03125*0.03125)
    };
    \legend{$k=1\text{, }k_p=2$, $k=2\text{, }k_p=3$, $k=2\text{, }k_p=2$, $k=3\text{, }k_p=4$, $\mathcal{O}(h)$, $\mathcal{O}(h^{2})$, $\mathcal{O}(h^{3})$}
  \end{semilogyaxis}
  \end{tikzpicture}
  \caption{$\Vert \cdot \Vert_{A_h^{P_2}}$-error for problem \eqref{discretepenaltyform2} with $k_g = k$, $\rho = h^{-1}$ and $\eta = h^{-2}$.}
  \label{fig:P2h}
\end{figure}

Finally we present results for problem \eqref{discretelagrangeform}. The exact Lagrange multiplier $\lambda$ is computed according to equation \eqref{charlambda}. We use a preconditioned MINRES solver with a block diagonal preconditioner as introduced in \cite{grossvectorlaplace} to solve the linear systems. In Figure \ref{fig:Lh} we present the error $\Vert \bu^* - \bu_h \Vert_{A_h^{L}}$ and in one case the error $\Vert \lambda - \lambda_h \Vert_{M}$, which is labeled with an $M$, for different choices of parameters and refinement levels (note that the two curves for $k_l=1$ are almost indistinguishable). 

\begin{figure}[ht!]
  \begin{tikzpicture}
  \def\vara{8.0}
  \def\varb{0.25}
  \def\varc{0.1}
  \begin{semilogyaxis}[ xlabel={Refinement level}, ylabel={Error}, ymin=5E-5, ymax=20, legend style={ cells={anchor=west}, legend pos=outer north east}, cycle list name=mark list ]
    \addplot table[x=level, y=uerror1] {Lh.dat};
    \addplot table[x=level, y=uerror2] {Lh.dat};
    \addplot table[x=level, y=merr2] {Lh.dat};
    \addplot table[x=level, y=uerror3] {Lh.dat};
    \addplot table[x=level, y=uerror4] {Lh.dat};
    \addplot[dashed,line width=0.75pt] coordinates { 
    (1,\vara) (2,\vara*0.5) (3,\vara*0.25) (4,\vara*0.125) (5,\vara*0.0625)(6,\vara*0.03125)
    };
    \addplot[dotted,line width=0.75pt] coordinates { 
      (1,\varb) (2,\varb*0.5*0.5) (3,\varb*0.25*0.25) (4,\varb*0.125*0.125) (5,\varb*0.0625*0.0625)(6,\varb*0.03125*0.03125)
    };
    \legend{$k=k_l=1\text{, }k_g=2$, $k=k_l=k_g=2$, $k=k_l=k_g=2\text{, }M$, $k=k_l=2\text{, }k_g=3$, $k= 2\text{, }k_l=1\text{, }k_g=3$, $\mathcal{O}(h)$, $\mathcal{O}(h^{2})$, $\mathcal{O}(h^{3})$}
  \end{semilogyaxis}
  \end{tikzpicture}
  \caption{$\Vert \cdot \Vert_{A_h^{L}}$-error and $\Vert \cdot \Vert_{M}$-error for problem \eqref{discretelagrangeform}.}
  \label{fig:Lh}
\end{figure}

We take $\rho = h^{-1}$ for superparametric finite elements ($k_g = k +1$) and $\rho = h$ for isoparametric finite elements ($k_g = k$).  For superparametric finite elements ($k_g=k + 1$) with $k_l=k$ we observe optimal $\mathcal{O}(h^k)$-convergence. For these cases the error $\Vert \lambda - \lambda_h \Vert_{M}$ has the same convergence order (not shown). However, \emph{iso}parametric quadratic finite elements ($k=k_g=2$) with $k_l=k$ results in optimal $\mathcal{O}(h^2)$-convergence for $\Vert \bu^*-\bu_h \Vert_{A_h^{L}}$ but \emph{sub}optimal $\mathcal{O}(h)$-convergence for $\Vert \lambda - \lambda_h \Vert_{M}$ (shown with label $M$ in the figure). This shows that the power $k_g-1$ in the term $h^{k_g-1}$ in Theorem \ref{energyerrorboundlh} is sharp and   superparametric finite elements ($k_g=k + 1$) are necessary to obtain an optimal order of convergence (for both primal variable and Lagrange multiplier). Taking $\rho = h^{-1}$ in this case  results in better than $\mathcal{O}(h)$-convergence but 
clearly  less than $\mathcal{O}(h^2)$-convergence (not shown).
For superparametric quadratic finite elements ($k=2$, $k_g=3$) with $k_l=1$ we observe (only) $\mathcal{O}(h)$-convergence. All these results are  in agreement with Theorem~\ref{energyerrorboundlh}, cf. Remark~\ref{Remdisc2}.

 A  drawback of the Lagrange multiplier method compared with the two penalty methods is the fact that (in our experience) the resulting saddle point system is (much) more difficult to solve.
The condition number of this matrix is typically very large, in particular for the case $\rho = h$. 

We did not derive $L^2$-error bounds and therefore do not present numerical results for $L^2$-errors. We note, however,  that for the cases of optimal $\mathcal{O}(h^k)$-convergence  in the energy norms we also have $\mathcal{O}(h^{k+1})$-convergence in the $L^2$-norm. In case of the Lagrange multiplier method \eqref{discretelagrangeform} we observe this  optimal $L^2$-norm convergence only for tangential error component, i.e. $\Vert \bP_h(\bu^* - \bu_h) \Vert_{L^2(\Gamma_h)}$. An analysis of $L^2$-norm convergence is left for future research.

\begin{remark} \label{rembound}
 \rm For the problem considered in this section we performed an experiment to see whether the $h^{-2}$ factor in the estimate \eqref{K1a} is sharp. We numerically computed
\[
c_h:=\min_{\bv_h \in \bV_{h,\Theta}^k}\frac{A_h^{L}(\bv_h, \bv_h)}{\|\bv_h \|_{H^1(\Gamma_h)}^2}
\]
for the parameter values $\rho =h $, $k=1$ and $k_g = 2$ as well as $k_g =1$. The results clearly indicate a $c_h \sim h^2$ behavior. 
\end{remark}

\bibliographystyle{siam}
\bibliography{literatur}{}

\end{document}